\documentclass[12pt]{amsart}
\usepackage{amssymb, tensor, fullpage, textcomp, pb-diagram, wasysym, wrapfig, comment, xspace, leftindex, xcolor, appendix, tensor, relsize}
\usepackage[Symbol]{upgreek}
\usepackage[mathscr]{euscript}
\usepackage[shortlabels]{enumitem}

\let\oldtocsection=\tocsection

\let\oldtocsubsection=\tocsubsection

\let\oldtocsubsubsection=\tocsubsubsection

\renewcommand{\tocsection}[2]{\hspace{0em}\oldtocsection{#1}{#2}}
\renewcommand{\tocsubsection}[2]{\hspace{1em}\oldtocsubsection{#1}{#2}}
\renewcommand{\tocsubsubsection}[2]{\hspace{2em}\oldtocsubsubsection{#1}{#2}}

\DeclareFontFamily{U}{fsy}{}
\DeclareFontShape{U}{fsy}{m}{n}{<->s*[.9]psyr}{}
\DeclareSymbolFont{der@m}{U}{fsy}{m}{n}
\DeclareMathSymbol{\der}{\mathord}{der@m}{182}

\newcommand{\tu}{T_\mathrm{U}}
\newcommand{\id}{\operatorname{id}}
\newcommand{\esp}{\mathrm{ES}_p}
\newcommand{\dlokk}{\dlo^\kappa_k}
\newcommand{\altT}{\mathrm{Alt}^*_{T,k}}
\newcommand{\lalt}{L_\mathrm{alt}}
\newcommand{\lvec}{L_\mathrm{vec}}

\newcommand{\nil}{\mathrm{Nil}}

\newcommand{\bil}{\mathrm{Bil}}

\newcommand{\wron}{\operatorname{Wr}}
\newcommand{\spn}{\operatorname{Span}}

\newcommand{\ips}{\mathrm{Hilb}_\R}
\newcommand{\cips}{\mathrm{Hilb}_\C}

\newcommand{\tfeq}{T^*_\mathrm{Feq}}

\newcommand{\dkmin}{D^\kappa_k(T)_{\_}}
\newcommand{\dkapk}{D^\kappa_k(T)}

\newcommand{\dcf}{\mathrm{DCF}}

\newcommand{\nset}{\{1,\ldots,n\}}
\newcommand{\gik}{(g_i)_{i < \kappa}}

\newcommand{\fik}{(f_i)_{i < \kappa}}
\newcommand{\sik}{(\sigma_i)_{i < \kappa}}
\newcommand{\hik}{(h_i)_{i < \kappa}}

\newcommand{\dlo}{\mathrm{DLO}}

\newcommand{\hyp}{\mathrm{Hyp}}
\newcommand{\rel}{\mathrm{Rel}}

\newcommand{\spcrngle}{\langle\hspace{.075cm},\hspace{.025cm}\rangle}
\newcommand{\spcbrckt}{[\hspace{.075cm},\hspace{.025cm}]}

\DeclareSymbolFont{imag@m}{OT1}{cmr}{m}{ui}
\DeclareMathSymbol{\imag}{\mathord}{imag@m}{105}

\newtheorem{theorem}{Theorem}[section]
\newtheorem*{theorem*}{Theorem}

\newtheorem*{qst*}{Question}
\newtheorem{thm}[theorem]{Theorem}
\newtheorem{proposition}[theorem]{Proposition}
\newtheorem*{proposition*}{Proposition}
\newtheorem*{fact*}{Fact}

\newtheorem*{Claim*}{Claim}
\newtheorem{fact}[theorem]{Fact}

\newtheorem{lemma}[theorem]{Lemma}
\newtheorem{corollary}[theorem]{Corollary}

\newtheorem*{thmA}{Theorem A}

\newtheorem*{thmB}{Theorem B}

\theoremstyle{definition}

\theoremstyle{remark}

\newcommand{\tp}{\operatorname{tp}}
\newcommand{\qftp}{\operatorname{qftp}}
\newcommand{\acf}{\mathrm{ACF}}

\newcommand{\rgroop}{(R; +, \trianglelefteq)}
\newcommand{\rgoup}{(\R;+,<)}

\newcommand{\rfield}{(\R;+,\times)}
\newcommand{\rcf}{\mathrm{RCF}}

\newcommand{\vvec}{\mathrm{Vec}}
\newcommand{\triv}{\mathrm{Triv}}

\newcommand{\Th}{\mathrm{Th}}
\ProvideTextCommandDefault{\cprime}{(U+042C)}

\newcommand{\Fraisse}{Fra\"iss\'e\xspace}
\newcommand{\Erdos}{Erd\H{o}s\xspace}

\newenvironment{claimproof}[1][\proofname]
               {
                 \proof[#1]
                 
               }
               {
                 \endproof
               }

\newcommand{\nip}{\mathrm{NIP}}
\newcommand{\ip}{\mathrm{IP}}
\newcommand{\nfop}{\mathrm{NFOP}}

\newcommand{\Cal}[1]{\ensuremath{\mathcal{#1}}}
\newcommand{\Sa}[1]{\ensuremath{\mathscr{#1}}}

\newcommand{\age}[1]{\ensuremath{\textup{Age{#1}}}}

\newcommand{\Z}{\mathbb{Z}}
\newcommand{\N}{\mathbb{N}}
\newcommand{\C}{\mathbb{C}}
\newcommand{\Q}{\mathbb{Q}}
\newcommand{\R}{\mathbb{R}}
\newcommand{\F}{\mathbb{F}}
\newcommand{\E}{\mathbb{E}}

\begin{document}
\title[]{Trace definability IV: higher arity notions}

\author{Erik Walsberg}
\email{erik.walsberg@gmail.com}
\begin{abstract}
Motivated by the ``composition theorems" of Chernikov-Hempel and Abd Aldaim-Conant-Terry we introduce $k$-trace definability between first order theories.
Any theory which is $k$-trace definable in a $\mathrm{NIP}$ theory is $k$-$\mathrm{NIP}$ and any theory which is $2$-trace definable in a stable theory is $2$-$\mathrm{NFOP}$.
All known examples of $k$-$\mathrm{NIP}$ theories are $k$-trace definable in $\mathrm{NIP}$ theories.
We show that for several of the main examples of $k$-$\mathrm{NIP}$ theories $T$ there is a $\mathrm{NIP}$ theory $T^*$ such that $T$ is the (unique up to a certain notion of equivalence) universal theory which is $k$-trace definable in $T^*$.
For example the theory of Hilbert space is the universal theory which is $2$-trace definable in $\mathrm{RCF}$, the theory of the generic class $k$ nilpotent Lie algebra over $\mathbb{F}_p$ is the universal theory which is $k$-trace definable in the theory of infinite $\mathbb{F}_p$-vector spaces,  the theory of the generic $k$-hypergraph is the universal theory which is $k$-trace definable in the theory of a set with two elements, and the theory of Uryshon space is the universal theory which is $2$-trace definable in $\Th\rgoup$.
We construct the universal theory $D_k(T)$ which is $k$-trace definable in an arbitrary theory $T$.
\end{abstract}

\maketitle
\section*{Introduction}
Among other things, we give a precise meaning to the following analogy.
\begin{center}
Infinite-dimensional Hilbert space : $\R$ :: The \Erdos-Rado graph : A two-element set.
\end{center}
This requires some explanation.
Chernikov and Hempel introduced a technique for showing that structures are $k$-$\nip$~\cite{hempel-chernkov, chernikov2025ndependentgroupsfieldsiii}.
It has been used to treat a series of examples.
They showed that a structure $\Sa M$ is $k$-$\nip$ when there is a $\nip$ reduct $\Sa M_0$ of $\Sa M$ and a family $\Cal E$ of $\Sa M$-definable functions $M^k \to M$ such that every $\Sa M$-definable subset of every $M^m$ is of the form 
\[
\{ (a_1, \ldots, a_m) \in M^m : (f_1(a_{i_{1, 1}}, \ldots, a_{i_{1, k}} ), \ldots , f_d(a_{i_{d, 1}}, \ldots, a_{i_{d, k}} ) ) \in Y\}
\]
for some $f_1, \ldots, f_d \in \Cal E$, indices $i_{j, \ell} \in \{1, \ldots, m\}$, and $\Sa M_0$-definable $Y \subseteq M^d$.
For example, let $\ips$ be the theory of infinite-dimensional real Hilbert space\footnote{This is the first order theory of Hilbert space, not the continuous theory.
It is shown in \cite{solovay-arthan-harrison} that a sentence $\varphi$ holds in $\ips$ iff it holds in $\R^d$ for $d$ at least the number of vector variables in $\varphi$.}.
Solovay, Arthan, and Harrison showed that any formula in vector variables $v_1, \ldots, v_m$ and scalar variables $c_1, \ldots, c_n$ is equivalent in $\ips$ to a formula of the form $\vartheta(\langle v \rangle, c_1, \ldots, c_n)$
where $\vartheta(x_1, \ldots, x_{m^2 + n})$ is a formula in the language of ordered fields and $\langle v \rangle$ is the tuple of inner products of the $v_i$~\cite{solovay-arthan-harrison}.
It follows that $\ips$ is $2$-$\nip$.
This is sharp as the formula $\langle v_1, v_2\rangle = 0$ is $\ip$.

\medskip
Let $T$ be an arbitrary theory.
Motivated by the Chernikov-Hempel technique we let $D_k(T)$ be the theory of two-sorted structures of the form $(P, \Sa M, f)$ where $P$ is an infinite set, $\Sa M \models T$, and $f$ is a generic function $P^k \to M$.
We show that $D_k(T)$ exists, is complete, and admits quantifier elimination relative to $T$.
It follows that $D_k(T)$ is strictly $k$-$\nip$ when $T$ is $\nip$.
Intuitively, $D_k(T)$ should be the universal $k$-$\nip$ theory which can be produced from $T$ using the Chernikov-Hempel technique.
This suggests, for example, that $\ips$ should be reducible to $D_2(\rcf)$ in some sense.

\medskip
We make this precise using two notions of reducibility introduced in \cite{trace1, trace2}.
We recall these notions. 
We say that $\Sa N$ {\bf trace defines} $\Sa M$ if one of the following equivalent conditions holds:
\begin{enumerate}[leftmargin=*]
\item Up to isomorphism, $M$ is a subset of some $N^n$ and every $\Sa M$-definable subset of every $M^m$ is of the form $Y \cap M^m$ for some $\Sa N$-definable $Y \subseteq N^{nm}$.
\item There is an injection $\uptau \colon M \hookrightarrow N^n$ for some $n$ such that every $\Sa M$-definable subset of every $M^m$ is a pullback via $\uptau$ of an $\Sa N$-definable subset of $N^{nm}$.
\end{enumerate}
There is also a coarser ``local" notion.
We say that $\Sa N$ {\bf locally trace defines} $\Sa M$ if for any $\Sa M$-definable sets $X_1, \ldots, X_d$ there is an injection $\uptau \colon M \hookrightarrow N^n$ for some $n$ such that each $X_i$ is the pullback via $\uptau$ of an $\Sa N$-definable set.
Several more definitions naturally follow.
A theory (locally) trace defines another if every (equivalently: some) model of the second is (locally) trace definable in a model of the first.
Two theories are (locally) trace equivalent if each is (locally) trace definable in the other.
Two structures are (locally) trace equivalent when their theories are.
Equivalently: two structures are (locally) trace equivalent if each is (locally) trace definable in an elementary extension of the other.
Classification-theoretic properties such as superstability, strong dependence, total transcendence, finiteness of various ranks, etc, are preserved under trace definability~\cite{trace1}.
Properties that are defined locally such as stability and $k$-$\nip$ are preserved under local trace definability.

\medskip
We show that any theory that can be produced from $T$ using the Chernikov-Hempel technique is locally trace definable in $D_k(T)$, and $D_k(T)$ is the unique theory with this property up to local trace equivalence.
If, as is usually the case, we only require finitely many functions, then we can drop ``locally" and ``local" here.
We need to explain what we mean by ``produced from $T$ using the Chernikov-Hempel technique".
Perhaps surprisingly, this notion is a higher arity version of trace definability.
We say that $\Sa N$ {\bf locally $k$-trace defines $\Sa M$} if there is a family $\Cal E$ of functions $M^k \to N$ such that every $\Sa M$-definable subset of every $M^m$ is of the form 
\[
\{ (a_1, \ldots, a_m) \in M^m : (f_1(a_{i_{1, 1}}, \ldots, a_{i_{1, k}} ), \ldots , f_d(a_{i_{d, 1}}, \ldots, a_{i_{d, k}} ) ) \in Y\}
\]
for some $f_1, \ldots, f_d \in \Cal E$, indices $i_{j, \ell} \in \{1, \ldots, m\}$, and $\Sa N$-definable $Y \subseteq N^d$.
We say that $\Sa N$ {\bf $k$-trace defines} $\Sa M$ if we may choose $\Cal E$ to be finite.
(Think of the $f \in \Cal E$ as $N$-valued invariants of elements of $M^k$.)
We extend these definitions to theories as above.
Local $k$-trace definability in a $\nip$ structure implies $k$-$\nip$.
As far as I know, every example of a $k$-$\nip$ structure is $k$-trace definable in a $\nip$ structure.
As the terminology suggests (local) $1$-trace definability is equivalent to (local) trace definability.
For any theories $T, T^*$ we have:

\begin{itemize}[leftmargin=*]
\item $T^*$ is $k$-trace definable in $T$ if and only if it is trace definable in $D_k(T)$, and $D_k(T)$ is the unique theory up to trace equivalence with this property.
\item $T^*$ is locally $k$-trace definable in $T$ if and only if it is locally trace definable in $D_k(T)$, and $D_k(T)$ is the unique theory up to local trace equivalence with this property.
\end{itemize}
We show that a number of the main examples of $k$-$\nip$ theories are trace equivalent to $D_k(T)$ for a $\nip$ theory $T$.
We give the main examples, described somewhat informally.

\begin{thmA}
\hspace{.00000000000000000000001cm}
\begin{itemize}[leftmargin=*]
\item The theory of infinite-dimensional real Hilbert space and the theory of infinite-dimensional complex Hilbert space are both trace equivalent to $D_2(\rcf)$.
\item The theory of the generic class $k$ nilpotent Lie algebra over $\F_p$ is trace equivalent to $D_k(\vvec_p)$, where $\vvec_p$ is the theory of infinite $\F_p$-vector spaces.
\item Let $p$ be an odd prime and $\bil_p$ be the theory of  two-sorted structures $(V, W, \beta)$ where $V, W$ are infinite $\F_p$-vector spaces and $\beta$ is an non-degenerate alternating bilinear form $V^2 \to W$.
Then $\bil_p$ is trace equivalent to $D_2(\vvec_p)$
\item Let $T$ be the theory of a characteristic zero field $\F$.
The theory of a generic infinite-dimensional class $k$ nilpotent Lie algebra over $\F$ is locally trace equivalent to $D_k(T)$.
The theory of an infinite-dimensional $\F$-vector space $V$ equipped with a generic $k$-linear alternating form $V^k \to \F$ is also locally trace equivalent to $D_k(T)$.
\item Let $\Sa R$ be an arbitrary expansion of $\rgoup$ and $T = \Th(\Sa R)$.
Consider the classical Uryshon space to be a two-sorted structure $(X, \Sa R, d)$ where $d$ is the metric $X^2 \to \R$.
Then the theory of $(X, \Sa R, d)$ is trace equivalent to $D_2(T)$.
\item The theory of a generic linear order on the $k$th power of an infinite set is trace equivalent to $D_k(\dlo)$.
\item The theory of the generic $k$-hypergraph  is trace equivalent to $D_k(\Th(\Sa F))$ where $\Sa F$ is an arbitrary fixed finite structure with at least two elements.
For any $m \ge 2$, the theory of the generic $km$-hypergraph is trace equivalent to $D_k(T)$ for $T$ the theory of the generic $m$-hypergraph. 
\end{itemize}
\end{thmA}
The reader familiar with $k$-$\nip$ structures may wonder whether the theory of infinite extra-special $p$-groups is trace equivalent to $D_2(T)$ for some $\nip$ theory $T$, as it is one of the first examples of a $2$-$\nip$ theory.
On the contrary, we show that this theory cannot be locally trace equivalent to $D_k(T)$ for any theory $T$.
This follows by combining several results.
\begin{enumerate}[leftmargin=*]
\item Any infinite extra-special $p$-group is trace equivalent to the disjoint union of an infinite $\F_p$-vector space with the \Erdos-Rado graph.
This shows in a precise sense that infinite extra-special $p$-groups are the simplest $\ip$ structures which are more complicated then a $\F_p$-vector spaces.
\item If $T$ is any theory with infinite models and $k \ge 2$ then $D_k(T)$ is not locally trace definable in the theory of a disjoint union of a stable structure with a binary structure.
This is proven by applying an ad-hoc property strictly intermediate between stability and $2$-$\nip$ which we call ``partition-wise stability".
\item The last claim of Theorem~A and Theorem~B below together show that $D_k(T)$ cannot locally trace define $\vvec_p$ where $T$ is the theory of some finite structure.
\end{enumerate}

\begin{thmB}
A theory admitting quantifier elimination in a bounded arity relational language cannot locally $k$-trace define a group which has an abelian subgroup of cardinality at least $n$ for every $n \in \N$.
\end{thmB}

Theorem~B covers the familiar examples of groups\footnote{Unless the reader happens to be familiar with free Burnside groups or Tarski monsters.}.
It should be possible to strengthen (3) to arbitrary infinite groups.
However, this appears to be substantially more difficult.
Theorem~B follows by adapting the proof of a theorem of Oger~\cite{oger}, which is itself an application of a strong form of Ramsey's theorem.

\subsection*{Acknowledgments}
This research was funded in part by the Austrian Science Fund (FWF) 10.55776/PAT1673125.



\subsection*{Conventions}
Throughout $m, n, k, l, d$ are natural numbers (including $0$) and $\kappa$ is a non-zero cardinal.
All languages, structures, and theories are first order and all theories are consistent, complete, and deductively closed unless stated otherwise.
These assumptions ensure that any $L$-theory $T$ has cardinality $|L| + \aleph_0$, so if $\kappa$ is an infinite cardinal then $|T| \le \kappa$ if and only if $|L| \le \kappa$.
Throughout ``definable" without modification means ``first order definable, possibly with parameters".
Given a language $L$, structure $\Sa M$, and $A \subseteq M$, we let $L(A)$ be the expansion of $L$ by constant symbols defining the elements of $A$.

\section{Basic facts about (local) $k$-trace definability}\label{section:k-trace def}
Fix $k\ge 1$.
Let $\Sa N, \Sa M$ be arbitrary structures and $T, T^*$ be arbitrary theories.
Suppose  $\Sa N$ eliminates quantifiers.
Then $\Sa N$ \textbf{locally $k$-trace defines} $\Sa M$ if there is a collection $\Cal E$ of functions $M^k \to N$ such that every $\Sa M$-definable subset of every $M^n$ is quantifier free definable in the two-sorted structure $(M, \Sa N,\Cal E)$.
We say that $\Sa N$ \textbf{$k$-trace defines} $\Sa M$ if we may additionally take $\Cal E$ to be finite, i.e. if there are functions $\uptau_1,\ldots,\uptau_m\colon M^k\to N$ such that every $\Sa M$-definable set is quantifier free definable in $(M, \Sa N,\uptau_1,\ldots,\uptau_m)$.
Note that if $\Sa N^*$ is another structure on $N$ which admits   quantifier elimination and is interdefinable with $\Sa N$  then $\Sa N^*$ (locally) $k$-trace defines $\Sa M$ if and only if $\Sa N$ (locally) $k$-trace defines $\Sa M$.
We extend the definitions to arbitrary structures by saying that an arbitrary structure $\Sa N$ (locally) $k$-trace defines $\Sa M$ if any structure which interdefinable with $\Sa N$ and admits quantifier elimination (locally) $k$-trace defines $\Sa M$.

\medskip
We say that $\Sa M$ is (locally) $k$-trace definable  in $T$ when $\Sa M$ is (locally) $k$-trace definable in a model of $T$, and $T^*$ is (locally) $k$-trace definable in $T$ when every model of $T^*$ is (locally) $k$-trace definable in $T$.
Proposition~\ref{prop:anudda} restates the definitions in a slightly different form.

\begin{proposition}
\label{prop:anudda}
Fix $k\ge 1$, let $\Sa N$ be an $L$-structure, and let $\Sa M$ be an $L^*$-structure.
Then the following are equivalent:
\begin{enumerate}
[leftmargin=*]
\item $\Sa M$ is $k$-trace definable in $\Sa N$.
\item There is a finite collection $\Cal E$ of functions $M^k \to N$ such that every $L^*(M)$-formula $\phi(x_1,\ldots,x_n)$ is equivalent to a formula of the form
\[
\vartheta(\uptau_1(x_{i_{1,1}},\ldots,x_{i_{1,k}}), \uptau_2(x_{i_{2,1}},\ldots,x_{i_{2,k}}),\ldots,\uptau_m(x_{i_{m,1}},\ldots,x_{i_{m,k}}))
\]
for  $\vartheta$ an $L(N)$-formula, $\uptau_1, \ldots, \uptau_m \in \Cal E$ and $i_{1,1},\ldots,i_{1,k},\ldots,i_{m,1},\ldots,i_{m,k}\in\{1,\ldots,n\}$.
\end{enumerate}
Furthermore the following are equivalent:
\begin{enumerate}[leftmargin=*]\setcounter{enumi}{2}
\item $\Sa M$ is locally $k$-trace definable in $\Sa N$.
\item There is a collection $\Cal E$ of functions $M^k \to N$ such that every $L^*(M)$-formula $\phi(x_1,\ldots,x_n)$ is equivalent to a formula of the form
\[
\vartheta(\uptau_1(x_{i_{1,1}},\ldots,x_{i_{1,k}}), \uptau_2(x_{i_{2,1}},\ldots,x_{i_{2,k}}),\ldots,\uptau_m(x_{i_{m,1}},\ldots,x_{i_{m,k}}))
\]
for  $\vartheta$ an $L(N)$-formula, $\uptau_1,\ldots,\uptau_m\in\Cal E$, and $i_{1,1},\ldots,i_{1,k},\ldots,i_{m,1},\ldots,i_{m,k}\in\{1,\ldots,n\}$.
\end{enumerate}
Here we take all formulas to be in the two sorted structure $(\Sa M, \Sa N,\Cal E)$.
\end{proposition}

If $\Cal E$ is as in (2), (4) then we say that $\Cal E$ {\bf witnesses} local $k$-trace definability, $k$-trace definability of $\Sa M$ in $\Sa N$, respectively.
We leave the proof of Proposition~\ref{prop:anudda} to the reader as it should be obvious upon reflection.
Note also that formulas of the form described in (2) are closed under boolean combinations, so it is enough to show that every $L^*(M)$-formula is equivalent to a boolean combination of such formulas.
We will often use this without mention.

\medskip
These definitions are motivated by work of Chernikov-Hempel and Abd Aldaim-Conant-Terry.
Let $L$ be an arbitrary language, $L^*$ be the expansion of $L$ by a collection $\Cal E$ of $k$-ary functions, $\Sa M^*$ be an $L^*$-structure, and $\Sa M$ be the $L$-reduct of $\Sa M^*$.
Let $\vartheta(y_1, \ldots, y_n)$ be an $L(M)$-formula.
Chernikov and Hempel show that if $\vartheta$ is $\nip$ in $\Sa M$ and $f_1, \ldots, f_n \in \Cal E$ then the formula
\[
\vartheta^* = \vartheta(f_1(x_{1, 1}, \ldots, x_{1, k}), \ldots, f_n(x_{n, 1}, \ldots, x_{n, k}))
\]
is $k$-$\nip$ in $\Sa M^*$~\cite[Thm.~3.24]{chernikov2025ndependentgroupsfieldsiii}.
Abd Aldaim, Conant, and Terry show that if $k = 2$ and $\vartheta$ is stable in $\Sa M$ then $\vartheta^*$ is $2$-$\nfop$\footnote{$k$-$\nfop$ is a proposed higher arity version of stability, see \cite{aldaim} for the definition.} in $\Sa M^*$~\cite[Thm.~2.17]{aldaim}.
Proposition~\ref{prop:ChHe} follows immediately.

\begin{proposition}\label{prop:ChHe}
Any theory which is locally $k$-trace definable in a $\nip$ theory is $k$-$\nip$.
Any theory which is locally $2$-trace definable in a stable theory is $2$-$\nfop$.
\end{proposition}


We now consider the relationship with trace definability.
The present notion of (local) $1$-trace definability is equivalent to the notion of (local) trace definability introduced in~\cite{trace1, trace2}.
We recall the previous definitions.
First, $\Sa M$ is trace definable in $\Sa N$ if either of the following equivalent conditions holds:
\begin{enumerate}[leftmargin=*]
\item Up to isomorphism $M$ is a subset of some $N^n$ and every $\Sa M$-definable subset of every $M^m$ is of the form $Y \cap M^m$ for some $\Sa N$-definable $Y \subseteq N^{nm}$.
\item There is an injection $\uptau \colon M \hookrightarrow N^n$ for some $n$ such that every $\Sa M$-definable subset of every $M^m$ is of the form \[\{ (a_1, \ldots, a_m) \in M^m : (\uptau(a_1), \ldots, \uptau(a_m)) \in Y \} \text{\quad for some $\Sa N$-definable $Y \subseteq N^{nm}$.}\]
\end{enumerate}
It is easy to see that (1) and (2) are equivalent.
If $\uptau$ is as in (2) and $\uptau_1, \ldots, \uptau_n$ are the functions $M \to N$ such that $\uptau = (\uptau_1, \ldots, \uptau_n)$ then $\Cal E = \{\uptau_1, \ldots, \uptau_n\}$ witnesses $1$-trace definability of $\Sa M$ in $\Sa N$.
Conversely, if $\{\uptau_1, \ldots, \uptau_n\}$ is a set of functions $M \to N$ witnessing $1$-trace definability of $\Sa M$ in $\Sa N$ then the map $\uptau \colon M \to N^n$ given by $\uptau = (\uptau_1, \ldots, \uptau_n)$ satisfies the condition in (2).
Furthermore $\Sa M$ is locally trace definable in $\Sa N$ when the following equivalent conditions hold.
\begin{enumerate}[leftmargin=*]\setcounter{enumi}{2}
\item For any $\Sa M$-definable set $X \subseteq M^{m}$ there is an injection $\uptau \colon M \hookrightarrow N^n$ for some $n$ such that \[X = \{ (a_1, \ldots, a_m) \in M^m : (\uptau(a_1), \ldots, \uptau(a_m)) \in Y \} \text{\quad for some $\Sa N$-definable $Y \subseteq N^{nm}$.}\] 
\item For any $\Sa M$-definable sets $X_1 \subseteq M^{m_1}, \ldots, X_d \subseteq M^{m_d}$ there is an injection $\uptau \colon M \hookrightarrow N^n$ for some $n$ such that each $X_i$ is of the form \[\{ (a_1, \ldots, a_{m_i}) \in M^{m_i} : (\uptau(a_1), \ldots, \uptau(a_{m_i})) \in Y \} \text{\quad for some $\Sa N$-definable $Y \subseteq N^{nm_i}$.}\] 
\item If $L$ is a finite relational language then any $\Sa M$-definable $L$-structure embeds into an $\Sa N$-definable $L$-structure.
\end{enumerate}
See \cite[Prop.~2.4]{trace2} for a proof of the equivalence of these definitions, and a proof that local trace definability is equivalent to local $1$-trace definability.
It is clear from these definitions that (local) trace definability is a transitive relation between structures and between theories.
So we say that two theories are {\bf (locally) trace equivalent} if each (locally) trace defines the other and two structures are (locally) trace equivalent when their theories are.
We now gather some basic facts about (local) $k$-trace definability.

\begin{lemma}
\label{lem:ka}
Suppose $\Sa N\models T$, $\Sa M\models T^*$ are structures, and $k\ge 1$.
\begin{enumerate}
[leftmargin=*]
\item If $\Sa M$ is locally $k$-trace definable in $\Sa N$ then this is witnessed by a collection of functions of cardinality at most $|T^*|$.
\item If every reduct of $\Sa M$ to a finite sublanguage is locally $k$-trace definable in $\Sa N$ then $\Sa M$ is locally $k$-trace definable in $\Sa N$.
\end{enumerate}
Suppose furthermore that $\Sa N$ admits quantifier elimination and $\Cal E$ is a collection of functions $M^k \to N$.
Then we have the following.
\begin{enumerate}[leftmargin=*]\setcounter{enumi}{2}
\item Suppose that every set which is zero-definable in $\Sa M$ is quantifier-free definable in $(M, \Sa N,\Cal E)$.
Then $\Cal E$ witnesses local $k$-trace definability of $\Sa M$ in $\Sa N$.
\item Suppose that $L^*$ is relational, that $T^*$ is an $L^*$-theory with quantifier elimination, and that $\{\alpha\in M^n : \Sa M\models R(\alpha)\}$ is quantifier-free definable in $(M, \Sa N,\Cal E)$ for all $n$-ary $R\in L^*$.
Then $\Cal E$ witnesses local $k$-trace definability of $\Sa M$ in $\Sa N$.
\item If $\Sa M$ admits quantifier elimination in a finite relational language then $\Sa M$ is locally $k$-trace definable in $\Sa N$ if and only if $\Sa M$ is $k$-trace definable in $\Sa N$.
\item Suppose that $\Cal F$ is a collection of functions $M^d \to N$, $d\le k$, such that every $\Sa M$-definable set is quantifier-free definable in $(M, \Sa N,\Cal F)$.
Then $\Sa M$ is locally $k$-trace definable in $\Sa N$.
If $\Cal F$ is finite then $\Sa M$ is $k$-trace definable in $\Sa N$.
\end{enumerate}
\end{lemma}

By (6) we could define (local) $k$-trace definability in terms of a collection $\Cal E$ of functions of arity at most $k$.
This would perhaps be more natural, but the notation would be even worse.
If $\Cal F$ is as in (6) then we will also say that $\Cal F$ witnesses local $k$-trace definability of $\Sa M$ in $\Sa N$.

\begin{proof}
After possibly Morleyizing, suppose that $\Sa N$ admits quantifier elimination.
Observe that (2), (3), and (4) are immediate from the definition and that (5) follows from (4).

\medskip
(1):
Suppose that $\Sa M$ is locally $k$-trace definable in $\Sa N$ and let $\Cal E$ be a collection of functions witnessing this.
For every set $X$ which is zero-definable in $\Sa M$ there is a finite subset $\Cal E_X$ of $\Cal E$ such that $X$ is quantifier-free definable in $(M, \Sa N, \Cal E_X)$.
Let $\Cal E$ be the union of the $\Cal E_X$.
Then $|\Cal E^*| \le |T^*|$ and an application of (3) shows that $\Cal E^*$ witnesses local $k$-trace definability of $\Sa M$ in $\Sa N$.
A similar argument yields (2).

\medskip
(6):
Let $\Cal E$ be the collection of functions $M^k \to N$ of the form $f(x_1,\ldots,x_k)=g(x_1,\ldots,x_d)$ for every $d$-ary $g\in\Cal F$.
Then $(M, \Sa N,\Cal F)$ is quantifier-free interdefinable with $(M, \Sa N,\Cal E)$, hence $\Cal E$ witnesses local $k$-trace definability of $\Sa M$ in $\Sa N$.
Furthermore $\Cal E$ is finite when $\Cal F$ is finite, hence $\Cal E$ witnesses $k$-trace definability of $\Sa M$ in $\Sa N$ when $\Cal F$ is finite.
\end{proof}

\begin{proposition}
\label{prop:trace-theories}
Fix theories $T,T^*$.
Then $T^*$ is (locally) $k$-trace definable in $T$ when some $T^*$-model is (locally) $k$-trace definable in a $T$-model.
\end{proposition}

We use Proposition~\ref{prop:trace-theories} constantly and often without mention.

\begin{proof}
Suppose that $\Cal E$ witnesses local $k$-trace definability of $\Sa M \models T^*$ in $\Sa N \models T$.
Let $\Sa O$ be an arbitrary model of $T$.
Let $(\Sa M^*, \Sa N^*, \Cal E^*)$ be an $|O|^+$-saturated elementary extension of $(\Sa M, \Sa N, \Cal E)$.
We may suppose that $\Sa O$ is an elementary substructure of $\Sa M^*$.
Let $\Cal F$ be the collection of functions given by restricting elements of $\Cal E^*$ to $O^k$.
Note that $\Cal F$ witnesses local $k$-trace definability of $\Sa O$ in $\Sa N^*$.
Finally, $\Cal F$ is finite when $\Cal E$ is finite, so this argument also covers $k$-trace definability.
\end{proof}

\begin{lemma}\label{lem:trace compose}
Let $\Sa M_1, \Sa M_2, \Sa M_3$ be structures, let $\Cal E_1$ a collection of functions $M_2^m \to M_1$ witnessing local $m$-trace definability of $\Sa M_2$ in $\Sa M_1$, and let $\Cal E_2$ be a collection of functions $M_3^n \to M_2$ witnessing local $n$-trace definability of $\Sa M_3$ in $\Sa M_2$.
Let $\Cal E_3$ be the collection of functions $M_3^{nm} \to M_1$ of the form
\[
f(x_1,\ldots,x_{nm}) = g( h_1(x_1,\ldots,x_n ) , \ldots, h_m(x_{n(m - 1) + 1},\ldots,x_{mn} ) )
\]
for $g\in \Cal E_1$ and $h_1,\ldots,h_m \in \Cal E_2$.
Then $\Cal E_3$ witnesses local $nm$-trace definability of $\Sa M_3$ in $\Sa M_1$.
\end{lemma}

Lemma~\ref{lem:trace compose} follows directly from the definition and is also left to the reader.
Note that $|\Cal E_3| = |\Cal E_1| |\Cal E_2|^m$, so in particular $\Cal E_3$ is finite when $\Cal E_1$ and $\Cal E_2$ are both finite.




\begin{proposition}
\label{prop:trace-basic}
Let $\Sa M_1, \Sa M_2, \Sa M_3$ be structures, $T_1, T_2, T_3$ be theories, and $m, n, k \ge 1$.
\begin{enumerate}[leftmargin=*]
\item\label{i1} If $\Sa M_3$ is (locally) $n$-trace definable in $\Sa M_2$ and $\Sa M_2$ is (locally) $m$-trace definable in $\Sa M_1$ then $\Sa M_3$ is (locally) $nm$-trace definable in $\Sa M_1$.
Furthermore if $\kappa, \eta \ge \aleph_0$, local $n$-trace definability of $\Sa M_3$ in $\Sa M_2$ is witnessed by $\le \kappa$ functions, and local $m$-trace definability of $\Sa M_2$ in $\Sa M_1$ is witnessed by $\le \eta$ functions, then local $nm$-trace definability of $\Sa M_3$ in $\Sa M_1$ is witnessed by $\le \kappa + \eta$ functions.
\item\label{i3} If $T_3$ is (locally) $n$-trace definable in $T_2$ and $T_2$ is (locally) $m$-trace definable in $T_1$ then $T_3$ is (locally) $nm$-trace definable in $T_1$.
\item\label{i5} (Local) $k$-trace definability between theories is invariant under (local) trace equivalence.
More formally:
Suppose that $T_i$ is \textup{(}locally\textup{)} trace equivalent to $T^*_i$ for $i = 1, 2$.
Then $T_1$ \textup{(}locally\textup{)} $k$-trace defines $T_2$ if and only if $T^*_1$ \textup{(}locally\textup{)} $k$-trace defines $T^*_2$.
\item\label{i6} 
If $T_2$ is locally $m$-trace definable in $T_1$ then $T_2$ is $n$-trace definable in $T_1$ for any $n > m$.
\end{enumerate}
\end{proposition}

\begin{proof}
Lemma~\ref{lem:trace compose} gives (\ref{i1}).
Note that (\ref{i3}) follows from (\ref{i1}) and the definitions.
Furthermore (\ref{i5}) follows from (\ref{i3}).
We prove (\ref{i6}).
By Lemma~\ref{lem:ka}(6) it is enough to treat the case when $n = m + 1$.
Suppose that $T_2$ is locally $m$-trace definable in $T_1$.
After possibly Morleyizing suppose $T_1$ admits quantifier elimination.
Fix $\Sa M_2 \models T_2$ with $|M_2| = |T_2|$.
Then $\Sa M_2$ is locally $m$-trace definable in some $\Sa M_1 \models T_1$.
Let $\Cal E$ be a collection of functions $M_2^m \to M_1$ witnessing this.
By Lemma~\ref{lem:ka}(1) we may suppose that $|\Cal E| = |T_2| = |M_2|$.
Let $\fik$ be an enumeration of $\Cal E$ and $(a_i)_{i <\kappa}$ be an enumeration of $M_2$.
Define $g \colon M_2^{m + 1} \to M_1$  by declaring $g(a_i, b_1, \ldots, b_m) = f_i(b_1, \ldots, b_m)$ for all $i < \kappa$ and $b_1, \ldots, b_m \in M_2$.
Any subset of any $M_2^k$ which is quantifier-free definable  in $(M_2, \Sa M_1, \Cal E)$ is also quantifier-free definable in $(M_2, \Sa M_1, g)$.
Hence $g$ witnesses $(m + 1)$-trace definability of $\Sa M_2$ in $\Sa M_1$.
\end{proof}




\begin{proposition}\label{prop:fhkd}
Suppose that $k \ge 2$, $L$ is a $k$-ary relational language, and $\Sa M$ is an $L$-structure with quantifier elimination.
Let $\Sa T$ be a structure in the language of equality with two elements.
Then $\Sa M$ is locally $k$-trace definable in $\Sa T$.
If $L$ is finite then $\Sa M$ is $k$-trace definable in $\Sa T$.
\end{proposition}



\begin{proof}
Let $\Sa T = \{p, q \}$.
For each $d$-ary $R \in L$ (including equality) let $\chi_R \colon M^d \to \{p, q\}$ be given by declaring $\chi_R(\alpha) = p$ if and only if $\Sa M \models R(\alpha)$ for all $\alpha \in M^d$.
An application of Lemma~\ref{lem:ka}(4) shows that $\{\chi_R : R \in L\}$  witnesses local $k$-trace definability of $\Sa M$ in $\Sa T$ and  witnesses $k$-trace definability when $L$ is finite.
\end{proof}

We finish by recalling two useful facts about trace definability.
The first is an immediate consequence of the elementary fact that the truth value of a quantifier-free formula is preserved under embeddings.

\begin{fact}\label{fact:trace embedd}
If $\Sa M$ and $\Sa N$ are structure in the same language, $\Sa M$ admits quantifier elimination, and $\uptau$ is an embedding $\Sa M \hookrightarrow \Sa N$, then $\uptau$ witnesses trace definability of $\Sa M$ in $\Sa N$.
\end{fact}

\medskip
Fact~\ref{fact:new hyp rel} summarizes some results about the theory $\hyp_k$ of the generic $k$-hypergraph and the theory $\rel_k$ of the generic $k$-ary relation.

\begin{fact}\label{fact:new hyp rel}
Fix $k \ge 2$.
\begin{enumerate}[leftmargin=*]
\item $\hyp_k$ and $\rel_k$ are trace equivalent.
\item Any structure admitting quantifier elimination in a finite $k$-ary relational language is trace definable in $\hyp_k$.
\item A theory $T$ is $(k - 1)$-$\ip$ if and only if it trace defines $\hyp_k$.
\end{enumerate}
\end{fact}

Fact~\ref{fact:new hyp rel} is proven in \cite[Lemma~2.4, Prop.~2.5]{trace1}.

\subsection{Disjoint unions}
We let $\Sa M_1 \sqcup \Sa M_2$ be the disjoint union of structure $\Sa M_1$ and $\Sa M_2$.
This is most naturally considered as a two-sorted structure.
We can also consider it as a one-sorted structure, in the same way that any finitely sorted structure can be considered as a one-sorted structure.
In particular this allows us to define (local) $k$-trace definability of $\Sa M_1 \sqcup \Sa M_2$ without having to define (local) $k$-trace definability between multi-sorted structures.
The trivial but important point is that a subset of $M^{m_1}_1 \times M^{m_2}_2$ is definable in $\Sa M_1 \sqcup \Sa M_2$ if and only if it is a finite union of sets of the form $X_1 \times X_2$ where each $X_i \subseteq M^{m_i}_i$ is definable in $\Sa M_i$.
Given theories $T_1, T_2$ we also let $T_1 \sqcup T_2$ be the theory of the disjoint union of a model of $T_1$ with a model of $T_2$.
Lemma~\ref{lem:k-dis} follows immediately.

\begin{lemma}\label{lem:k-dis}
Let $\Sa N$, $\Sa M_1$, $\Sa M_2$ be arbitrary structures, $T, T_1, T_2$ be arbitrary theories, and fix $k \ge 1$.
Then $\Sa M_1 \sqcup \Sa M_2$ is (locally) $k$-trace definable in $\Sa N$ if and only if $\Sa M_1$ and $\Sa M_2$ are both (locally) $k$-trace definable in $\Sa N$.
Hence $T_1 \sqcup T_2$ is (locally) $k$-trace definable in $T$ if and only if $T_1$ and $T_2$ are both (locally) $k$-trace definable in $T$.
\end{lemma}

Some structures decompose into disjoint unions up to trace equivalence.
Suppose that $P$ is a model-theoretic property and $\Sa C$ is a structure such that a theory $T$ trace defines $\Sa C$ if and only if $T$ has $P$.
If $\Sa M$ is trace equivalent to $\Sa M_1 \sqcup \Sa C$, then we can say in a precise sense that $\Sa M$ is the simplest structure which is more complicated then $\Sa M_1$ and has $P$.
For example, consider the case when $\Sa C$ is the unary relational structure with domain the Cantor set and unary relations defining all clopen subsets.
Then $T$ is not totally transcendental if and only if $T$ trace defines $\Sa C$~\cite[Prop.~3.3]{trace1}.
We showed in \cite[Prop.~4.5]{trace2} that $(\Z; +)$ is trace equivalent to $(\Q; +) \sqcup \Sa C$, so $(\Z; +)$ is the simplest structure which is more complicated then $(\Q; +)$ and is not totally transcendental.
See Proposition~\ref{prop:final final} below for another example.

\section{The universal theory which is $k$-trace definable in $T$}

\subsection{The theory of $\kappa$ generic $k$-ary functions taking values in a model of  $T$}\label{section:genk}
Throughout this section we fix a  language $L$, an $L$-theory $T$, a cardinal $\kappa \ge 1$, and an integer $k \ge 1$.
We also suppose that $T$ has a model with $\ge 2$ elements.
Let $D^\kappa_k(T)_{\_}$ be the theory of  two sorted structures of the form $(P, \Sa M,\fik)$ where $\Sa M\models T$, $P$ is an infinite set, and each $f_i$ is a function $P^k \to M$.
Note that if $T$ is definitionally equivalent to $T^*$ then $D^\kappa_k(T^*)_{\_}$ is definitionally equivalent to $D^\kappa_k(T^*)_{\_}$.
We define the model companion $\dkapk$ of $D^\kappa_k(T)_{\_}$ relative to $T$.

\medskip
We will let $D_k(T)=D^1_k(T)$ and $D^\kappa(T)=D^\kappa_1(T)$.
The theory $D^\kappa(T)$ was introduced, and its basic properties were developed, in~\cite{trace3}.

\medskip
A {\bf $(\kappa, k)$-function structure} consists of sorts $P, M$ and $\kappa$ functions $P^k \to M$.
When $\kappa$ is finite we say that a $(\kappa, k)$-function structure $(P, M; \fik)$ is {\bf rich} if whenever $(X, Y; \sik)$ is a finite $(\kappa, k)$-function structure and $X^* \subseteq X$, then any embedding of $(X^*, Y; \sik)$ into $(P, M; \fik)$ extends to an embedding of $(X, Y; \sik)$ into $(P, M; \fik)$.
In general we say that a $(\kappa, k)$-function structure is rich when every reduct to a finite sublanguage is rich.
If $(P, M; \fik)$ is rich then we will also say that $\fik$ is rich.
It should be clear from the definition that the class of rich $(\kappa, k)$-function structures is first order with a $\forall\exists$ axiomization.

\begin{lemma}\label{lem:exists P}
Any $(\kappa, k)$-function structure $(P, M; \fik)$ extends to a rich $(\kappa, k)$-function structure $(Q, M; \gik)$.
\end{lemma}

This is a routine closure argument, so we leave some details to the reader.

\begin{proof}
We need to produce a set $Q$ containing $P$ and a rich collection $\gik$ of functions $Q^k \to M$ such that each $g_i$ extends $f_i$.
Working inductively, we construct for each $n$ a set $Q_n$ and a family $(g^n_i)_{i < \kappa}$ of functions $Q^k_n \to M$ such that we have the following for all $n$:
\begin{enumerate}[leftmargin=*]
\item $Q_0 = P$ and $Q_n \subseteq Q_{n + 1}$,
\item $g^0_i = f_i$ and $g^{n + 1}_i$ extends $g^n_i$ for each $i < \kappa$,
\item If $I \subseteq \kappa$ is finite, $(X, Y; (\sigma_i)_{i \in I})$  is a finite $(\kappa, k)$-function structure, and $X^* \subseteq X$, then any embedding of $(X^*, Y; (\sigma_i)_{i \in I})$ into $(Q_n, M; (g^n_i)_{i \in I})$ extends to an embedding of $(X, Y; (\sigma_i)_{i \in I})$ into $(Q_{n + 1}, M; (g^{n + 1}_i)_{i \in I})$.
\end{enumerate}
We then let $Q$ be the union of the $Q_n$ and for each $i < \kappa$ let $g_i$ be the function $Q^k \to M$ whose restriction to each $Q^k_n$ is $g^n_i$.
Finally, note that $(Q, M; \gik)$ is rich.
\end{proof}

\begin{lemma}\label{lem:crucial embedd}
Suppose that
\begin{enumerate}[leftmargin=*]
\item $(P, M; \fik)$ is a $(\kappa, k)$-function structure, 
\item  $(Q, N; \gik)$ is a rich $\max(|P|, |M|, \kappa)^+$-saturated $(\kappa, k)$-function structure, 
\item $P_0$ is a subset of $P$,
\item and $\eta$ is an embedding of $(P_0, M; \fik)$ into $(Q, N; \gik)$.
\end{enumerate}
Then $\eta$ extends to an embedding of $(P, M; \fik)$ into $(Q, N; \gik)$.
\end{lemma}

\begin{proof}
By saturation it is enough to fix finite $I \subseteq \kappa$ and show that any finite substructure $(X, Y; (f_i)_{i \in I})$ of $(P, M; (f_i)_{i \in I})$ admits an embedding into $(Q, N; (g_i)_{i \in I})$ which agrees with $\eta$ on both $X \cap P_0$ and $Y$.
This follows by richness.
\end{proof}

We now define $\dkapk$.
Let $D^\kappa_k(T)_{\_}$ be the theory of  two sorted structures of the form $(P, \Sa M,\fik)$ where $\Sa M\models T$, $P$ is an infinite set, and each $f_i$ is a function $P^k \to M$.
Let $L^\kappa_k$ be the language of $D^\kappa_k(T)_{\_}$.
Let $D^\kappa_k(T)$ be the theory of models $(P, \Sa M, \fik)$ of $\dkmin$ with $\fik$ rich.
Note that if $T$ is $\forall\exists$ then $D^\kappa_k(T)_{\_}$ and $D^\kappa_k(T)$ are both $\forall\exists$.
Furthermore, note that if $T$ and $T^*$ are definitionally equivalent then $D^\kappa_k(T)$ and $D^\kappa_k(T^*)$ are also definitionally equivalent.
Finally, note that $D^1_1(T)$ is mutually interpretable with $T$.


\begin{lemma}\label{lem:triv obs}
A  map between models of $\dkmin$ is an embedding if and only if it gives both an embedding of $L$-structures and an embedding of $(\kappa, k)$-function structures.
\end{lemma}

Lemma~\ref{lem:triv obs} is trivial but useful.

\begin{lemma}\label{lem:universal}
Suppose that $\Sa M \models T$, $O$ is a set, $f$ is a function $O^k \to M$, and $(P, \Sa N, g)$ is a $\max(|M|, |O|^+)$-saturated model of $D_k(T)$ such that $\Sa M$ is a substructure of $\Sa N$.
Then there is an injection $\uptau \colon O \hookrightarrow
P$ such that $\uptau$ and the inclusion $\Sa M \hookrightarrow \Sa N$ together give an embedding $(O, \Sa M, f) \hookrightarrow (P, \Sa N, g)$.
\end{lemma}

Lemma~\ref{lem:universal} generalizes to the case $\kappa > 1$, but we only need this case.

\begin{proof}
Fix $q \in O$ and $p \in P$ such that $f(q, \ldots, q) = g(p, \ldots, p)$.
Let $\eta$ be the map  $(\{q\}, \Sa M, f) \to (P, \Sa N, g)$ that sends $q$ to $p$ and agrees with the inclusion $M \hookrightarrow N$ on $M$.
Note that $\eta$ is an embedding.
An application of Lemma~\ref{lem:crucial embedd} shows that $\eta$ extends to an embedding $\eta^*$ of $(O, M; f)$ into $(P, N; g)$.
Lemma~\ref{lem:triv obs} shows that $\eta^*$ gives an embedding of $(O, \Sa M, f)$ into $(P, \Sa N, g)$.
Let $\uptau$ be the map $O \to P$ given by $\eta^*$.
\end{proof}

\begin{lemma}
\label{lem:blow}
Suppose that  $T$ has quantifier elimination.
Then $D^\kappa_k(T)$ is the model companion of $D^\kappa_k(T)_{\_}$ and $D^\kappa_k(T)$ is complete and admits quantifier elimination.
\end{lemma}

\begin{proof}
We first show that $\dkapk$ is the model companion of $\dkmin$.
It follows by Lemma~\ref{lem:exists P} that any model of $\dkmin$ embeds into a model of $\dkapk$.
Hence it is enough to show that a fixed model $(P, \Sa M, \fik)$ of $D^\kappa_k(T)$ is existentially closed in the class of models of $\dkmin$.
Let $(P^*,  \Sa M^*, \fik)$ be a model of $\dkmin$ extending $(P, \Sa M, \fik)$.
Let $(Q, \Sa N, \gik)$ be a $\max(\kappa, |P^*|, |M^*|)^+$-saturated elementary extension of $(P, \Sa M, \fik)$.
It suffices to show that there is an embedding $\eta$ of $(P^*,  \Sa M^*, \fik)$ into $(Q, \Sa N, \gik)$ which fixes every element of both $P$ and $M$.
Our assumption of quantifier elimination for $T$ ensures that $\Sa M^*$ is an elementary extension of $\Sa M$.
Hence by saturation we may suppose that $\Sa M^*$ is an elementary substructure of $\Sa N$.
Now observe that $(P, M^*; \fik)$ is a $(\kappa, k)$-function structure and a substructure of $(Q, N; \gik)$.
By Lemma~\ref{lem:crucial embedd} the inclusion $(P, M^*; \fik) \hookrightarrow (Q, N; \gik)$ extends to an embedding $\eta$ of $(P^*, M^*; \fik)$ into $(Q, N; \gik)$.
Finally, $\eta$ gives an embedding of $(P^*, \Sa M^*, \fik)$ into $(Q, \Sa N, \gik)$ by Lemma~\ref{lem:triv obs}. 

\medskip
We now show that $\dkapk$ admits quantifier elimination.
As $\dkapk$ is model complete it is enough to show that $\dkapk_\forall$ has the amalgamation property~\cite[Thm.~8.4.1]{Hodges}.
Note that a model of $\dkapk_\forall$ is a structure $(P, \Sa M, \fik)$ where $\Sa M \models T_\forall$ and $(P, M; \fik)$ is a $(\kappa, k)$-function structure.
Fix a model $(P, \Sa M, \fik)$ of $\dkapk_\forall$ and let $(P_j, \Sa M_j, (f^j_i)_{i < \kappa})$ be another model of $\dkapk_\forall$ extending $(P, \Sa M, \fik)$ for $j = 1,2$.
Now $T_\forall$ has the amalgamation property as $T$ admits quantifier elimination so there is a model $\Sa N \models T_\forall$ and embeddings $\Sa M_j \hookrightarrow \Sa N$ for $j = 1, 2$ such that the resulting square commutes.
We may suppose that $P_1 \cap P_2 = P$.
Let $P_\cup = P_1 \cup P_2$.
For each $i < \kappa$ let $g_i$ be some function $P^k_\cup \to N$ which agrees with $f^j_i$ on $P^k_j$ for $j = 1, 2$.
This is possible as $f^1_i$ and $f^2_i$ agree on $P^k$.
So $(P_\sqcup, \Sa N, \gik) \models \dkapk_\forall$.
Finally, the embeddings $\Sa M_j \hookrightarrow \Sa N$ and the inclusion $P_j \hookrightarrow P_\sqcup$ give an embedding $(P_j, \Sa M_j, (f^j_i)_{i < \kappa}) \hookrightarrow (Q, \Sa N, \hik)$ for $j = 1, 2$.
This gives the desired amalgamation.

\medskip
It remains to show that $D^\kappa_k(T)$ is complete.
By model completeness it is enough to show that any two models of $\dkapk$ jointly embed into a third.
As $T$ is complete any two models of $T$ embed into a third, so we can use a slight variation of the argument given in the previous paragraph which we leave to the reader.
\end{proof}

\subsection{$\dkapk$ and trace definability}
We discuss the relationship between $\dkapk$ and trace definability.

\begin{lemma}
\label{lem:blown}
Let $T$ be an arbitrary theory.
Then $D^\kappa_k(T)$ is complete and any formula $\phi(x_1,\ldots,x_n,y_1,\ldots,y_m)$ in $D^\kappa_k(T)$ with each $x_i,y_j$ a variable of the first sort, second sort, respectively is equivalent to a boolean combination of formulas in the language of equality in the variables $x_1,\ldots,x_n$ and formulas of the form
\[
\vartheta(f_1(x_{i_{1,1}},\ldots,x_{i_{1,k}}),\ldots,f_d(x_{i_{d,1}},\ldots,x_{i_{d,k}}),y_1,\ldots,y_m )
\]
for an $L$-formula $\vartheta(z_1,\ldots,z_{d + m})$, $f_1,\ldots,f_d \in L^\kappa_k \setminus L$, and indices $i_{j, l}\in\nset$.
Furthermore if $\eta$ is an embedding between models of $D^\kappa_k(T)$ and the induced embedding between $L$-structures is elementary then $\eta$ is elementary.
\end{lemma}

Lemma~\ref{lem:blown} shows in particular that any $L^\kappa_k$-formula with all variables of the first sort is equivalent in $D^\kappa_k(T)$ to an $L$-formula.
Hence if $(P, \Sa M, \fik)\models D^\kappa_k(T)$ then the induced structure on $M$ is interdefinable with $\Sa M$.

\begin{proof}
Morleyize and apply Lemma~\ref{lem:blow}.
We leave the details to the reader.
\end{proof}




Suppose that  $(P, \Sa M,\fik)\models D^\kappa_k(T)$.
Let $\Sa P$ be the structure induced on $P$ by $(P, \Sa M, \fik)$.
Each $f_i$ is a surjection $P^k\to M$, hence we may consider $M$ to be an imaginary sort of $\Sa P$.
So $\Sa P$ is bi-interpretable with $(P, \Sa M, \fik)$.
We will often replace $(P, \Sa M, \fik)$ with $\Sa P$.
By Lemma~\ref{lem:blown} every definable subset of $P^n$ is a boolean combination of sets definable in the language of equality and sets of the form
\[
\{ (p_1,\ldots,p_n)\in P^n : \Sa M\models \vartheta(f_{i_1}(p_{i_{1,1}},\ldots p_{i_{1,k}}),\ldots,f_{i_m}(p_{i_{m,1}},\ldots,p_{i_{m,k}} )) \}
\]
for some $m$-ary $L(M)$-formula $\vartheta$, $i_1, \ldots, i_m < \kappa$, and indices $i_{j, l}$ from $\nset$.

\medskip
We now consider the relationship between $\dkapk$ and local $k$-trace definability in $T$.

\begin{lemma}\label{lem:c}
Suppose that either $T$ has infinite models or $T$ has a finite model with $\ge 2$ elements and $k \ge 2$.
If $\kappa$ is finite then $D_k^\kappa(T)$ is $k$-trace definable in $T$.
If $\kappa$ is infinite then $\dkapk$ is locally $k$-trace definable in $T$, and this is witnessed by $\le \kappa$ functions.
\end{lemma}

\begin{proof}
Fix $(P, \Sa M, \fik) \models \dkapk$.
Let $\id_M$ be the identity $M \to M$.
First suppose that $T$ has infinite models.
Then we may suppose that $|P| \le |M|$.
Let $\iota$ be an arbitrary injection $P \hookrightarrow M$.
Lemma~\ref{lem:blown} shows that $\fik$, $\iota$, $\id_M$ together witness local $k$-trace definability of $(P, \Sa M, \fik)$ in $\Sa M$.
Note that $\iota$ handles formulas in the language of equality with variables ranging over $P$.
Now suppose that $T$ has a finite model with $\ge 2$ elements and that $k \ge 2$.
Let $a, b$ be distinct elements of $M$ and let $g \colon P^2 \to \{a, b\}$ be given by declaring $g(p, p^*) = a$ if and only if $p = p^*$.
Now observe that $\fik$, $g$, $\id_M$ witness local $k$-trace definability of $(P, \Sa M, \fik)$ in $\Sa M$.
(Now $g$ handles the formulas that $\iota$ handled in the previous case.)
\end{proof}

Theorem~\ref{thm;dock} is the key result on $D^\kappa_k(T)$.

\begin{thm}
\label{thm;dock}
Suppose that either $T$ has infinite models or $T$ has a finite model with $\ge 2$ elements and $k \ge 2$.
Then the following holds for any theory $T^*$ and $\kappa \ge \aleph_0$.
\begin{enumerate}
[leftmargin=*]
\item  $T$ $k$-trace defines $T^*$ if and only if $D_k(T)$ trace defines $T^*$.
\item $T$ locally $k$-trace defines $T^*$ if and only if $D_k(T)$ locally trace defines $T^*$.
\item $T$ locally $k$-trace defines $T^*$, and this is witnessed by at most $\kappa$ functions, if and only if $D^\kappa
_k(T)$ trace defines $T^*$.
\item If $\kappa\ge |T^*|$ then $T$ locally $k$-trace defines $T^*$ if and only if $D^\kappa_k(T)$  trace defines $T^*$.
\end{enumerate}
Furthermore we have the following for any theory $T$.
\begin{enumerate}[label=(\alph*),leftmargin=*]
\item $D_k(T)$ is the unique theory modulo trace equivalence satisfying (1) for all theories $T^*$.
\item  $D_k(T)$ is the unique theory modulo local trace equivalence with (2) for all theories $T^*$.
\item  $D^\kappa_k(T)$ is the unique theory modulo trace equivalence satisfying (3) for all theories $T^*$.
\item If $\kappa\ge |T|$ then $D^\kappa_k(T)$ is the unique theory of cardinality $\le\kappa$ modulo trace equivalence satisfying (4) for all theories $T^*$ of cardinality $\le\kappa$.
\end{enumerate}
\end{thm}

It follows in particular that $D^\kappa_k(T)$ is trace equivalent to $D_k(T)$ when $\kappa$ is finite.
For this reason we will only consider $\dkapk$ when $\kappa$ is either $1$ or infinite.

\begin{proof}
Note that (1) describes the class of theories that are trace definable in $D_k(T)$ and hence characterizes $D_k(T)$ up to trace equivalence.
Hence (a) follows from (1).
Likewise (b), (c), and (d) follows from (2), (3), and (4), respectively.
We prove (1) - (4).

\medskip
Lemma~\ref{lem:c} shows that $D_k(T)$ is $k$-trace definable in $T$.
So any theory that is (locally) trace definable in $D_k(T)$ is (locally) $k$-trace definable in $T$ by Proposition~\ref{prop:trace-basic}.
This gives the right to left directions of both (1) and (2).
Lemma~\ref{lem:c} also shows that $D^\kappa_k(T)$ is locally $k$-trace definable in $T$ and that this is witnessed by $\kappa$ functions.
Hence if $T^*$ is trace definable in $\dkapk$ then $T^*$ is locally $k$-trace definable in $T$, and this is witnessed by $\kappa$ functions.
The right to left implication of (3) and (4) follow.

\medskip
We prove the left to right implication of (1).
Suppose that $T^*$ is $k$-trace definable in $T$.
Fix $\Sa O \models T^*$, $\Sa M \models T$, and $\uptau_1, \ldots, \uptau_m \colon O^k \to M$ witnessing $k$-trace definability of $\Sa O$ in $\Sa M$.
Let $(P, \Sa N, f) \models D_k(T)$ be $\max(|M|, |O|)^+$-saturated and suppose that $\Sa M$ is an elementary submodel of $\Sa N$.
An application of Lemma~\ref{lem:universal} yields an injection $\upeta_i \colon O \hookrightarrow P$ for each $i = 1, \ldots, m$ such that 
\[
\uptau_i(a_1, \ldots, a_k) = f(\upeta_i(a_1), \ldots, \upeta_i(a_k)) \quad \text{for all $a_1, \ldots, a_k \in O$.}
\]
It follows that the $\upeta_i$ witness trace definability of $\Sa O$ in $(P, \Sa N, f)$.

\medskip
The left to right implication of (4) follows from the left to right implication of (3) and Lemma~\ref{lem:ka}(1).
We now prove the left to right implication of (3).
Suppose $\Sa O \models T^*$, $\Sa M \models T$, and $\fik$ is a collection of functions $O^k \to M$ witnessing local $k$-trace definability of $\Sa O$ in $\Sa M$.
Applying Lemma~\ref{lem:exists P} we obtain a set $P$ containing $O$ and a rich family $\gik$ of functions $P^k \to M$ such that each $g_i$ extends $f_i$.
So $(P, \Sa M, \gik) \models D^\kappa_k(T)$.
If $X \subseteq O^n$ is $\Sa O$-definable then $X$ is quantifier-free definable in $(O, \Sa M, \fik)$, hence we have $X = Y \cap O^n$ for a $(P, \Sa M, \gik)$-definable subset $Y \subseteq P^n$.
Hence the inclusion $O \hookrightarrow P$ witnesses trace definability of $\Sa O$ in $(P, \Sa M, \gik)$.

\medskip
We finally prove the left to right direction of (2).
By (3) it is enough to show that $D^\lambda_k(T)$ is locally trace definable in $D_k(T)$ for any cardinal $\lambda \ge 1$.
By Lemma~\ref{lem:ka}(2) it is enough to prove this in the case when $\lambda$ is finite, and this follows from (1) and Lemma~\ref{lem:c}.
\end{proof}

\begin{corollary}
\label{cor;dock}
Fix theories $T,T^*$ and suppose $\kappa \ge |T|,|T^*|$.
Suppose that either $T$ has infinite models or $T$ has a finite model with $\ge 2$ elements and $k \ge 2$.
Then we have the following.
\begin{enumerate}[leftmargin=*]
\item If $T$ trace defines $T^*$ then $D_k(T)$ trace defines $D_k(T^*)$.
If $T$ is trace equivalent to $T^*$ then $D_k(T)$ is trace equivalent to $D_k(T^*)$.
\item If $T$ locally trace defines $T^*$ then $D_k(T)$ locally trace defines $D_k(T^*)$.
If $T$ is locally trace equivalent to $T^*$ then $D_k(T)$ is locally trace equivalent to $D_k(T^*)$.
\item If $T$ locally trace defines $T^*$ then $D^\kappa_k(T)$  trace defines $D^\kappa_k(T^*)$.
If $T$ is locally trace equivalent to $T^*$ then $D^\kappa_k(T)$ is trace equivalent to $D^\kappa_k(T)$.
\end{enumerate}
\end{corollary}

\begin{proof}
The second claim of each of (1), (2), and (3) follows from the first.
If $T$ trace defines $T^*$ then every theory that is $k$-trace definable in $T^*$ is also $k$-trace definable in $T$, and hence $D_k(T)$ trace defines $D_k(T^*)$ by Theorem~\ref{thm;dock}(1).
This gives (1).
A similar argument gives (2).
We prove (3).
Suppose that $T$ locally trace defines $T^*$.
Now $D^\kappa_k(T^*)$ is locally $k$-trace definable in $T^*$, hence $D^\kappa_k(T^*)$ is locally $k$-trace definable in $T$.
We have $|D^\kappa_k(T^*)| = \kappa$, so $D^\kappa_k(T^*)$ is trace definable by $\dkapk$ by Theorem~\ref{thm;dock}(4).
\end{proof}

\begin{corollary}\label{cor:small}
If $T$ is $\nip$ then $\dkapk$ is $k$-$\nip$ and $(k - 1)$-$\ip$ for any $k \ge 2$ and $\kappa \ge 1$.
\end{corollary}

\begin{proof}
Corollary~\ref{cor;dock} and Proposition~\ref{prop:ChHe} together show that $\dkapk$ is $k$-$\nip$.
We show that $\dkapk$ is $(k - 1)$-$\ip$.
Proposition~\ref{prop:fhkd} and Theorem~\ref{thm;dock} together show that $\dkapk$ trace defines $\hyp_k$.
Apply Fact~\ref{fact:new hyp rel}(3).
Alternatively, it is easy to see that if $(P, \Sa M, \fik) \models \dkapk$ then $f_i(x_1, \ldots, x_k) = y$ is a $(k - 1)$-$\ip$ formula for any $i < \kappa$.
\end{proof}

\subsection{Composition theorems}\label{section:comp}
We consider theories of the form $D_m(D_n(T))$, $D_k(D^\kappa(T))$, etc.
We only defined these operations on one-sorted theories, and $D^\kappa_k(T)$ is two-sorted.
We can rectify this by considering any $(P, \Sa M, \fik) \models D^\kappa_k(T)$ to be a one-sorted structure with domain $P \cup M$ and unary relations defining $P$ and $M$.

\begin{proposition}\label{prop:compose combine}
Let $T$ be a theory which has a model with $\ge 2$ elements, $\kappa$ and $\lambda$ be infinite cardinals, and $k, m, n \ge 2$.
Then we have the following.
\begin{enumerate}[leftmargin=*]
\item $D^\kappa(D_k(T))$ is trace equivalent to $D^\kappa_k(T)$.
\item $D_m(D_n(T))$ is trace equivalent to $D_{mn}(T)$.
\end{enumerate}
If $T$ has infinite models then we additionally have the following.
\begin{enumerate}[leftmargin=*]\setcounter{enumi}{2}
\item $D^\lambda(D^\kappa(T))$ is trace equivalent to $D^{\lambda + \kappa}(T)$.
\item $D_k(D^\kappa(T))$ is trace equivalent to $\dkapk$.
\end{enumerate}
\end{proposition}

In Sections~\ref{section:examplesI} and \ref{section:dkex} we will give examples of $n$-$\nip$ theories $T^*$ that are trace equivalent to $D_n(T)$ for a $\nip$ theory $T$.
Note that (2) shows that in this case (local) $m$-trace definability in $T^*$ is equivalent to (local) $mn$-trace definability in $T$ for all $m \ge 1$.

\begin{proof}
(1):
Local trace definability of $D^\kappa(D_k(T))$ in $D_k(T)$ is witnessed by $\kappa$ functions and $k$-trace definability of $D_k(T)$ in $T$ is witnessed by three functions, hence local $k$-trace definability of $D^\kappa(D_k(T))$ in $T$ is witnessed by $\kappa$ functions.
Thus $D^\kappa(D_k(T))$ is trace definable in $D^\kappa_k(T)$ by Theorem~\ref{thm;dock}(3).

\medskip
We show that $\dkapk$ is trace definable in $D^\kappa(D_k(T))$.
Fix $(P, \Sa M, \fik) \models D^\kappa_k(T)$.
If $i_1, \ldots, i_n < \kappa$ are distinct then $(P, \Sa M, f_{i_1}, \ldots, f_{i_n})$ is a model of $D^n_k(T)$ and is hence trace definable in $D_k(T)$.
It follows that $(P, \Sa M, \fik) \models D^\kappa_k(T)$ is locally trace definable in $D_k(T)$, and that this is witnessed by $\kappa$ functions.
Hence $\dkapk$ is trace definable in $D^\kappa(D_k(T))$.

\medskip
(2):
Note first that $D_n(T)$ is $n$-trace definable in $T$ and $D_m(D_n(T))$ is $m$-trace definable in $D_n(T)$, hence $D_m(D_n(T))$ is $mn$-trace definable in $T$ by Proposition~\ref{prop:trace-basic}, hence $D_m(D_n(T))$ is trace definable in $D_{mn}(T)$.

\medskip
We show that $D_{mn}(T)$ is trace definable in $D_m(D_n(T))$.
By Theorem~\ref{thm;dock} it suffices to show that $D_{mn}(T)$ is $m$-trace definable in $D_n(T)$.
Fix $(P, \Sa M, f) \models D_{mn}(T)$.
Let $g$ be the function $(P^m)^n \to M$ given by 
\[
g((a_1, \ldots, a_m), \ldots, (a_{n(m - 1) + 1}, \ldots, a_{mn})) = f(a_1, \ldots, a_{mn}) \quad \text{for all $a_1, \ldots, a_{mn} \in P$.}
\]
By Lemma~\ref{lem:exists P} there is a set $Q$ containing $P^m$ and a function $g^* \colon Q^n \to M$ extending $g$ such that $(Q, \Sa M, g^*) \models D_n(T)$.
Finally note that the inclusion $P^m \hookrightarrow Q$ and the identity $M \to M$ together witness $m$-trace definability of $(P, \Sa M, f)$ in $(Q, \Sa M, g^*)$.

\medskip
(3):
This is proven in~\cite[Prop.~2.5]{trace3}.

\medskip
(4):
An argument similar to that given in (1) shows that $D_k(D^\kappa(T))$ is trace definable in $\dkapk$.
We show that $\dkapk$ is trace definable in $D_k(D^\kappa(T))$.
It suffices to show that $\dkapk$ is $k$-trace definable in $D^\kappa(T)$.
Fix $(P, \Sa M, \fik) \models \dkapk$.
By Lemma~\ref{lem:exists P} there is a set $Q$ extending $P^k$ and a family $\gik$ of functions $Q \to M$ such that each $g_i$ extending $f_i$ and $(Q, \Sa M, \gik) \models D^\kappa(T)$.
Finally, note that the inclusion $P^k \hookrightarrow Q$ and the identity $M \to M$ together witness $k$-trace definability of $(P, \Sa M, \fik)$ in $(Q, \Sa M, \gik)$.
\end{proof}

\begin{proposition}
\label{thm:factor}
Let $T,T^*$ be theories with infinite models and $m,n\ge 1$.
Then we have the following.
\begin{enumerate}[leftmargin=*]
\item $T^*$ is $mn$-trace definable in $T$ if and only if there is a theory $T^{**}$ such that $T^{*}$ is $m$-trace definable in $T^{**}$ and $T^{**}$ is $n$-trace definable in $T$. 
\item $T^*$ is locally $mn$-trace definable in $T$ if and only if there is a theory $T^{**}$ such that $T^{*}$ is locally $m$-trace definable in $T^{**}$ and $T^{**}$ is locally $n$-trace definable in $T$. 
\item $T^*$ is locally $mn$-trace definable in $T$ if and only if there is a theory $T^{**}$ such that $T^{*}$ is $m$-trace definable in $T^{**}$ and $T^{**}$ is locally $n$-trace definable in $T$.
\item $T^*$ is locally $mn$-trace definable in $T$ if and only if there is a theory $T^{**}$ such that $T^{*}$ is locally $m$-trace definable in $T^{**}$ and $T^{**}$ is $n$-trace definable in $T$.
\end{enumerate}
\end{proposition}

Let $[T]$ be the trace equivalence class of a theory $T$.
For each $k\ge 1$ let $R_k, R^\mathrm{loc}_k$ be the binary relation on trace equivalence classes given by declaring  $R_k([T^*],[T]), R^\mathrm{loc}_k([T^*],[T])$ when $T^*$ is $k$-trace definable, locally $k$-trace definable in $T$, respectively.
Proposition~\ref{thm:factor} is equivalent to the following equalities, where $\circ$ is the usual composition of binary relations.

\[
R_{mn} = R_m\circ R_n  \qquad \text{and} \qquad R^\mathrm{loc}_{mn} = R^\mathrm{loc}_{m}\circ R^\mathrm{loc}_{n}   = R_{m}\circ R^\mathrm{loc}_{n} = R^\mathrm{loc}_{m}\circ R_{n}
\]



\begin{proof}
The right to left implication of each item follows by Proposition~\ref{prop:trace-basic}.
Suppose that $T^*$ is $mn$-trace definable in $T$.
Then $T^*$ is trace definable in $D_{mn}(T)$.
Proposition~\ref{prop:compose combine} shows that $D_{mn}(T)$ is trace equivalent to $D_m(D_n(T))$.
So $T^*$ is $m$-trace definable in $D_n(T)$ and $D_n(T)$ is $n$-trace definable in $T$.
This yields (1).
A similar argument gives (2).

\medskip
We now fix $k \ge 2$ and suppose that $T^*$ is locally $k$-trace definable in $T$.
We show that:
\begin{enumerate}[label=(\alph*),leftmargin=*]
\item There is a theory $T^{**}$ such that $T$ is $k$-trace definable in $T^{**}$ and $T^{**}$ is locally trace definable in $T^{*}$.
\item There is a theory $T^{**}$ such that $T$ is locally trace definable in $T^{**}$ and $T^{**}$ is $k$-trace definable in $T^{*}$.
\end{enumerate}
Fix a cardinal $\kappa \ge |T^*|$.
Then $T^*$ is trace definable in $D^\kappa_k(T)$ by Theorem~\ref{thm;dock}(4).
By Proposition~\ref{prop:compose combine}(2) $D^\kappa_k(T)$ is trace equivalent to both $D^\kappa(D_k(T))$ and $D_k(D^\kappa(T))$.
So (a) holds with $T^{**} = D^\kappa(T)$ and (b) holds with $T^{**} = D_k(T)$.

\medskip
Now (3) follows from (1) and (a) as
\[
R^\mathrm{loc}_{mn} = R_{mn} \circ R^\mathrm{loc}_1 = (R_m \circ R_n) \circ R^\mathrm{loc}_1 = R_m \circ (R_n \circ R^\mathrm{loc}_1) = R_m \circ R^\mathrm{loc}_n.
\]
Likewise (4) follows from (1) and (b).
\end{proof}





We now give an application of the composition theorems.
Below we give examples of a family $\Cal T$ of theories and a fixed $\nip$ theory $T$ such that every member of $\Cal T$ is trace equivalent to some $\dkapk$.
Theorem~\ref{thm:distinct} largely describes trace definability between elements of $\Cal T$.

\begin{thm}\label{thm:distinct}
Let $T$ be a $\nip$ theory, $m, n \ge 1$, and $\kappa, \eta$ be cardinals such that each of $\kappa, \eta$ is either $1$ or $> |T|$.
Then $D^\kappa_m(T)$ trace defines $D^\eta_n(T)$ if and only if either $n < m$ or $n = m$ and $\eta \le \kappa$, hence $D^\kappa_m(T)$ is trace equivalent to $D^\eta_n(T)$ if and only if $m = n$ and $\kappa = \eta$.
\end{thm}

It is natural to ask if this generalizes to the case when $T$ is $k$-$\nip$.
By the proof below it would be enough to generalize the composition theorem of Chernikov and Hempel, see the discussion above Proposition~\ref{prop:ChHe}.
Fact~\ref{fact:distinct} is proven in~\cite[Thm.~2.7]{trace3}.

\begin{fact}\label{fact:distinct}
If $T$ is $\infty$-$\nip$ and $|T| \le \kappa < \eta$ then $D^\kappa(T)$ does not trace define $D^\eta(T)$.
\end{fact}

\begin{proof}[Proof of Theorem~\ref{thm:distinct}]
It is enough to prove the first claim.
If $n=m$ and $\eta\le\kappa$ then $D^\eta_n(T)$ is a reduct of $D^\kappa_m(T)$.
If $n<m$ then every structure that is locally $n$-trace definable in $T$ is also $m$-trace definable in $T$ by Proposition~\ref{prop:trace-basic}(\ref{i6}), hence $D^\eta_n(T)$ is trace definable in $D_m(T)$ by Theorem~\ref{thm;dock}.
This gives the right to left direction.

\medskip
Corollary~\ref{cor:small} and preservation of $m$-$\nip$ under trace definability together show that $D^\kappa_m(T)$ cannot trace define $D^\eta_n(T)$ when $m < n$.
Suppose that $m = n$.
By Proposition~\ref{prop:compose combine}(2) $D^\kappa_m(T)$, $D^\eta_m(T)$  is trace equivalent to $D^\kappa(D_m(T))$, $D^\eta(D_m(T))$, respectively.
Now $D_m(T)$ is $m$-$\nip$, so $D^\kappa(D_m(T))$ does not trace define $D^\eta(D_m(T))$ when $\kappa < \eta$ by Fact~\ref{fact:distinct}.
\end{proof}

\section{Examples}\label{section:examples}

\subsection{The trivial theory, dense linear orders, and the generic $m$-hypergraph}\label{section:examplesI}
We consider $D^\kappa_k(T)$ where $T$ is $\triv$, $\dlo$, or the theory $\hyp_m$ of the generic $m$-hypergraph.
Let $\rel^\kappa_k$ be the model companion of the theory of a set equipped with $\kappa$ relations, each of arity $k$.
This theory exists, is complete, and admits quantifier elimination.
This follows from~\cite[Thm.~3.7, Cor.~B.3]{jera}, but it is also an elementary exercise in quantifier elimination.

\begin{proposition}\label{prop:finite case}
Let $\Sa M$ be a finite structure with at least two elements, $T = \Th(\Sa M)$, $k \ge 1$, and $\kappa$ be either $1$ or an infinite cardinal.
Then $\dkapk$ is trace equivalent to $\rel^\kappa_k$ and $D_k(T)$ is trace equivalent to $\hyp_k$.
\end{proposition}

\begin{proof}
The second claim follows from the first by Fact~\ref{fact:new hyp rel}(1).
Note that $\Sa M$ is mutually interpretable with the structure in the language of equality with two elements.
So we may suppose that $\Sa M$ is this structure by Corollary~\ref{cor;dock}(1).
Now $\dkapk$ is the model companion of the theory of structures $(P, M; \fik)$ where $P$ is an infinite set, $M$ is a set with two elements, and each $f_i$ is a function $P^k \to M$.
It easily follows that $\dkapk$ is bi-interpretable with $\rel^\kappa_k$.
\end{proof}

Proposition~\ref{prop:hyper} now follows by combining Fact~\ref{fact:new hyp rel}(1) with Propositions~\ref{prop:finite case} and \ref{prop:compose combine}.

\begin{proposition}\label{prop:hyper}
Fix $m \ge 2$, $k \ge 1$, and let $\kappa$ be either $1$ or an infinite cardinal.
Then $D^\kappa_k(\hyp_m)$ is trace equivalent to $\rel^\kappa_{km}$.
In particular $D_k(\hyp_m)$ is trace equivalent to $\hyp_{km}$.
\end{proposition}

Our next goal is to prove analogues of Proposition~\ref{prop:hyper} for $\triv$ and $\dlo$.
We first describe a general construction on \Fraisse classes with disjoint amalgamation which will be used to construct the relevant theories.
A theory $T$ is {\bf algebraically trivial} if the algebraic closure of any subset $A$ of a model of $T$ is $A$.
Fact~\ref{fact:fraisse} contains the basic facts about \Fraisse classes and limits that we need.

\begin{fact}\label{fact:fraisse}
Let $T$ be a (necessarily incomplete) universal theory in a finite relational language such that the finite models of $T$ form a \Fraisse class $\Cal C$ with limit $\Sa M$.
Then $\Th(\Sa M)$ is the model companion of $T$.
Furthermore $\Th(\Sa M)$ is algebraically trivial if and only if $\Cal C$ has disjoint amalgamation.
\end{fact}

Let $L$ be a finite relational language.
Fix $k \ge 1$ and a cardinal $\kappa \ge 1$.
Let $L^{(k, \kappa)}$ be the language containing a $dk$-ary relation $R^*_i$ for every $i < \kappa$ and $d$-ary $R\in L$ other than equality.
Given an $L^{(k, \kappa)}$-structure $\Sa M$ and $i < \kappa$ we let $\Sa M[L, i]$ be the $L$-structure with domain $M^k$ given by declaring 
$$\Sa M[L, i]\models R\left((\alpha_1,\ldots,\alpha_k),(\alpha_{k+1},\ldots,\alpha_{2k}),\ldots,(\alpha_{(d-1)k+1},\ldots,\alpha_{dk})\right)$$
if and only if $\Sa M\models R^*_i(\alpha_1,\ldots,\alpha_{dk})$ for all $d$-ary $R\in L$ other than equality  and elements $\alpha_1,\ldots,\alpha_{dk}\in M$.
Given a class $\Cal C$ of finite $L$-structures let $B^\kappa_k(\Cal C)$ be the class of finite $L^{(k, \kappa)}$-structures such that $\Sa M[L, i] \in \Cal C$ for all $i < \kappa$.

\begin{lemma}\label{lem:dag}
Suppose that $\Cal C$ is a \Fraisse class of $L$-structures with disjoint amalgamation.
Fix $k, m \ge 1$.
Then $B^m_k(\Cal C)$ is a \Fraisse class with disjoint amalgamation.
\end{lemma}

\begin{proof}
The case when $k = 1$ is obvious and observed in \cite{Bodirsky2014}.
Clearly $B^m_k(\Cal C) = B^1_k(B^m_1(\Cal C))$, so we may suppose that $m = 1$.
Write $B_k(\Cal C) = B^1_k(\Cal C)$ and $\Sa M[L] = \Sa M[L, 0]$.
Note that $B_k(\Cal C)$ is a hereditary class as $\Cal C$ is a hereditary class.
We show that $B_k(\Cal C)$ satisfies disjoint amalgamation.
Fix $\Sa A, \Sa B_1, \Sa B_2 \in B_k(\Cal C)$.
Suppose that $\Sa A$ is a substructure of both $\Sa B_1$ and $\Sa B_2$, and $A = B_1 \cap B_2$.
Now $\Sa A[L]$ is a substructure of $\Sa B_i[L]$ for $i = 1, 2$ and $A^k = B^k_1 \cap B^k_2$.
We apply disjoint amalgamation for $\Cal C$.
Let $\Sa D \in \Cal C$ be an extension of $\Sa A[L]$ and $f_i$ be an embedding $\Sa B_i[L] \hookrightarrow \Sa D$ for $i = 1, 2$ such that $f_1(B^k_1) = A^k = f_2(B^k_2)$.
We may suppose that $\Sa B_i[L]$ is a substructure of $\Sa D$ and that $f_i$ is the inclusion for $i = 1, 2$.
As $\Cal C$ is hereditary we may suppose that $\Sa D$ has domain $B^k_1 \cup B^k_2$.
Let $\Sa D^* \in \Cal C$ be an extension of $\Sa D$ with cardinality $|B_1 \cup B_2|^k$.
Let $\iota$ be a bijection $D^* \to (B_1 \cup B_2)^k$ that is the identity on $B^k_1 \cup B^k_2$ and let $\Sa E$ be the pushforward of $\Sa D^*$ by $\iota$.
Finally let $\Sa B_3$ be the $L^{(k, 1)}$-structure with domain $B_1 \cup B_2$ such that $\Sa B_3[L] = \Sa E$.
Then $\Sa B_1, \Sa B_2$ are substructures of $\Sa B_3$ with intersection $\Sa A$.
A similar argument shows that $B_k(\Cal C)$ satisfies the disjoint joint amalgamation property.
\end{proof}

We now define the associated operation on certain theories.

\begin{lemma}\label{lem:new wink}
Suppose that $L$ is a finite relational language and $T$ is an algebraically trivial $L$-theory with quantifier elimination.
Fix $k \ge 1$ and a cardinal $\kappa \ge 1$.
Let $L^{(k, \kappa)}$ be as above and let $B^\kappa_k(T)$ be the model companion of the theory of $L^{(k, \kappa)}$-structures $\Sa M$ such that each $\Sa M[L, i]$ is a model of $T$.
Then we have the following.
\begin{enumerate}[leftmargin=*]
\item $B^\kappa_k(T)$ exists, is complete, admits quantifier elimination, and is algebraically trivial.
\item $B^\kappa_k(T)$ is locally $k$-trace definable in $T$ and is $k$-trace definable in $T$ when $\kappa$ is finite.
\end{enumerate}
\end{lemma}

\begin{proof}
(1):
The infinite case easily reduces to the finite case.
We leave this to the reader and suppose that $\kappa$ is finite.
Let $\Cal C = \age(\Sa M)$ for any $\Sa M \models T$, so $\Cal C$ is a \Fraisse class and $T$ is the theory of the limit of $\Cal C$.
Then $B^\kappa_k(\Cal C)$ is a \Fraisse class with disjoint amalgamation by Lemma~\ref{lem:dag}.
Let $B^\kappa_k(T)$ be the theory of the \Fraisse limit of $B^\kappa_k(\Cal C)$ and apply Fact~\ref{fact:fraisse}.

\medskip
(2):
Let $\Sa M$ be a model of $B^\kappa_k(T)$ and let $\Sa N$ be an $|M|^+$-saturated model of $T$.
For each $i < \kappa$ let $\uptau_i$ be an injection $M^k \hookrightarrow N$ which gives an embedding $\Sa M[L, i] \hookrightarrow \Sa N$.
Let $\uptau \colon M \hookrightarrow N$ be an arbitrary injection.
By quantifier elimination for $\Sa M$ and Lemma~\ref{lem:ka}(4) $\uptau$ and the $\uptau_i$ together witness local $k$-trace definability of $\Sa M$ in $\Sa N$.
Finally, $\uptau$ and the $\uptau_i$ witness $k$-trace definability of $\Sa M$ in $\Sa N$ when $\kappa$ is finite.
\end{proof}

We now consider $\dlo$.
Fix a cardinal $\kappa$ which is either $1$ or infinite and fix $k \ge 1$.
We let $\dlo^\kappa_k$ be the model companion of the theory of a set $M$ equipped with $\kappa$ linear orders on $M^k$.
So $\dlo^\kappa_k$ is $B^\kappa_k(\dlo)$.
By Lemma~\ref{lem:new wink}(1) $\dlo^\kappa_k$ exists, is complete, and admits quantifier elimination.

\begin{proposition}\label{prop:dlo case}
Suppose that $\kappa$ is either $1$ or an infinite cardinal.
Then $D^\kappa_k(\dlo)$ is trace equivalent to $\dlo^\kappa_k$.
\end{proposition}

\begin{proof}
Lemma~\ref{lem:new wink}(2) shows that $\dlo^\kappa_k$ is locally $k$-trace definable in $\dlo$ when $\kappa \ge \aleph_0$ and is $k$-trace definable in $\dlo$ when $\kappa = 1$.
Hence $D^\kappa_k(\dlo)$ trace defines $\dlo^\kappa_k$ by Theorem~\ref{thm;dock}. 
We now show that $\dlokk$ trace defines $D^\kappa_k(\dlo)$.
Let $\trianglelefteq$ be a binary relation on a set $X$.
Then $\trianglelefteq$ is a {\bf linear preorder} if it is a transitive binary relation such that we either have $a \trianglelefteq b$ or $b \trianglelefteq a$ for all $a,b \in X$.
If $\trianglelefteq$ is a preorder then $(a \trianglelefteq b) \land (b \trianglelefteq a)$ defines an equivalence relation on $X$ and $\trianglelefteq$ pushes forward to a linear order on $X/E$.

\begin{Claim*}\label{clm:1}
Suppose that $\trianglelefteq$ is a linear preorder on a set $X$.
Then there are linear orders $\le, \le^*$ on $X$ such that for any $a,b \in X$ we have $a\trianglelefteq b$ if and only if $(a \le b) \land (a \le^* b)$.
\end{Claim*}


\begin{claimproof}
Let $E$ be the equivalence relation associated to $\trianglelefteq$ as above.
Let $\le$ be an arbitrary linear order extending $\trianglelefteq$.
Then each $E$-class is convex with respect to $\trianglelefteq$.
Let $\le^*$ be the linear order given by reversing $\le$ on each $E$-class but maintaining the same order between $E$-classes.
It is easy to see that this works.
\end{claimproof}
Now fix $(P, M; \trianglelefteq, \fik) \models D^\kappa_k(\dlo)$.
In particular $(M; \trianglelefteq) \models \dlo$.
Let $\Sa P$ be the induced structure on $P$.
It suffices to show that $\Sa P$ is trace definable in $\dlokk$.
Let $Q = P \times \{1, 2\}$.
For each $i, j < \kappa$ let $h_{i, j}$ be some function $Q^k \to M$ such that 
\[
h_{i, j}((a_1, i_1), \ldots, (a_k, i_k)) =
\begin{cases}
f_i(a_1, \ldots, a_k) & \text{when $i_1 = \ldots = i_k = 1$ } \\
f_j(a_1, \ldots, a_k) & \text{when $i_1 = \ldots = i_k = 2$ }
\end{cases}
\]
Let $\leq_{i, j}$ be the linear preorder on $Q^k$ given by declaring $b \leq_{i, j} b^*$ when $h_{i, j}(b) \trianglelefteq h_{i, j}(b^*)$ for each $i, j < \kappa$.
Applying the claim, for each $i, j < \kappa$ let $\leq^1_{i, j}, \leq^2_{i, j}$ be linear orders on $Q^k$ such that $\leq_{i, j}$ is the intersection of $\leq^1_{i, j}$ and $\leq^2_{i, j}$.
Now let $\Sa N = (N; (\leq^1_{i, j})_{i, j < \kappa}, (\leq^2_{i, j})_{i, j < \kappa})$ be a model of $\dlo^{2\kappa}_k$ extending $(Q; (\leq^1_{i, j})_{i, j < \kappa}, (\leq^2_{i, j})_{i, j < \kappa})$.
Given $i = 1, 2$ let $\uptau_i \colon P \hookrightarrow N$ be given by $\uptau_i(a) = (a, i)$.
Tracing through, we have
\[
f_i(a_1, \ldots, a_k) \trianglelefteq f_j(b_1, \ldots, b_k) \quad \Longleftrightarrow \quad \bigwedge_{\ell = 1, 2} (\uptau_1(a_1), \ldots, \uptau_1(a_k)) \leq^\ell_{i, j} (\uptau_2(b_1), \ldots, \uptau_2(b_k))
\]
for any $a_1, b_1, \ldots, a_k, b_k \in P$ and $i, j < \kappa$.
It follows that $\uptau_1, \uptau_2$ witness trace definability of $\Sa P$ in $\Sa N$.
This handles the case when $\kappa$ is infinite.
Suppose $\kappa = 1$.
We have shown that $\Sa P$ is trace definable in $\dlo^2_k$.
Finally, Lemma~\ref{lem:new wink}(2) shows that $\dlo^2_k$ is trace equivalent to $\dlo^1_k$ as $\dlo^2_k$ is definitionally equivalent to $B^2_1(\dlo^1_k)$.
\end{proof}

We prove the analogous result for $\triv$.
Let $E^\kappa_k$ be the model companion of the theory of a set $M$ equipped with $\kappa$ equivalence relations on $M^k$.
Now $E^1_1$ is the theory of an equivalence relation with infinitely many classes, each of which is infinite.
Furthermore $E^\kappa_k$ is $B^\kappa_k(E^1_1)$.
Hence Lemma~\ref{lem:new wink}(1) shows that $E^\kappa_k$ exists, is complete, and admits quantifier elimination.

\begin{proposition}\label{prop:dkk triv}
Suppose that $\kappa$ is either $1$ or an infinite cardinal.
Then $D^\kappa_k(\triv)$ is trace equivalent to $E^\kappa_k$.
\end{proposition}

This is very similar to the proof of Proposition~\ref{prop:dlo case}, so we omit some details.
One can produce a common generalization of these results, but this is more trouble than it is worth at present.

\begin{proof}
Of course $E^1_1$ is mutually interpretable with $\triv$, so $D^\kappa_k(E^1_1)$ is trace equivalent to $D^\kappa_k(\triv)$.
Lemma~\ref{lem:new wink}(1) shows that $E^\kappa_k$ is locally $k$-trace definable in $E^1_1$ and is $k$-trace definable in $E^1_1$ when $\kappa = 1$.
Hence $E^\kappa_k$ is trace definable in $D^\kappa_k(\triv)$.

\medskip
Now let $(P, M; \fik) \models D^\kappa_k(\triv)$ and $\Sa P$ be the induced structure on $P$.
Let $Q = P \times \{1, 2\}$.
For each $i, j < \kappa$ let $h_{i, j} \colon Q^k \to M$ be as in the proof of Proposition~\ref{prop:dlo case} and let $E_{i, j}$ be the equivalence relation on $Q^k$ given by declaring $E_{i, j}(b, b^*)$ when $h_{i, j}(b) = h_{i, j}(b^*)$ for any $b, b^* \in Q^k$.
Now let $\Sa N$ be a model of $E^{\kappa}_k$ extending $(Q; (E_{i, j})_{i, j < \kappa})$.
Given $i = 1, 2$ let $\uptau_i \colon P \hookrightarrow N$ be given by $\uptau_i(a) = (a, i)$.
Now for any $a_1, b_1, \ldots, a_k, b_k \in P$ and $i, j < \kappa$ we have
\[
f_i(a_1, \ldots, a_k) = f_j(b_1, \ldots, b_k) \quad \Longleftrightarrow \quad E_{i, j}(\uptau_1(a_1), \ldots, \uptau_1(a_k), \uptau_2(b_1), \ldots, \uptau_2(b_k)).
\]
Hence $\uptau_1, \uptau_2$ witness trace definability of $(P, M; \fik)$ in $\Sa N$.
\end{proof}

\subsection{Two general lemmas}\label{section:dkex}
We next consider algebraic examples.
We first prove two general lemmas.

\begin{lemma}\label{lem;f 0}
Fix a structure $\Sa O\models T$, a theory $T^*$, and $k\ge 2$.
Suppose that for every set $P$ and function $f \colon P^k \to O$ there is $\Sa M \models T^*$, an $\Sa M$-definable set $X \subseteq M^n$, an $\Sa M$-definable function $g \colon X^k \to M^m$, and elements $(a^i_p : (i,p) \in \{1, \ldots, k\} \times P)$ of $X$ such that 
\begin{enumerate}[leftmargin=*]
\item $O$ is a subset of $M^m$,
\item $\Sa M$ trace defines $\Sa O$ via the inclusion $O \hookrightarrow M^m$,
\item and $g(a^1_{p_1}, \ldots, a^k_{p_k}) = f(p_1, \ldots, p_k)$ for all $p_1, \ldots, p_k \in P$.
\end{enumerate}
Then $T^*$ trace defines $D_k(T)$.
\end{lemma}

\begin{proof}
After possibly Morleyizing we suppose that $\Sa O$ admits elimination in a relational language $L$.
By Lemma~\ref{lem:exists P} there is a set $P$ and function $f \colon P^k \to O$ such that $(P, \Sa O, f) \models D_k(T)$.
Let $\Sa M$, $X$, $g$, and $(a^i_p : (i,p) \in \{1, \ldots, k\} \times P)$ be as above.
We show that $\Sa M$ trace defines $(P, \Sa O, f)$.
By Fact~\ref{fact:trace embedd} and quantifier elimination for $(P, \Sa O, f)$ it is enough to show that $(P, \Sa O, f)$ embeds into a structure definable in $\Sa M$.
It is easy to see that $\Sa O$ is a substructure of  an $\Sa M$-definable $L$-structure $\Sa O^*$ with domain $M^m$.
Let $h\colon (X^k)^k \to M^m$ be given by 
\[
h((a^1_1,\ldots,a^1_k),\ldots,(a^k_1,\ldots,a^k_k)) = g(a^1_1,a^2_2,\ldots,a^k_k).
\]
Let $\uptau\colon P \to X^k$ be given by $\uptau(p) = (a^1_p, \ldots, a^k_p)$.
Now $\uptau$ and the inclusion $O \hookrightarrow M^m$ together give an embedding $(P, \Sa O, f) \hookrightarrow (X^k, \Sa O^*, h)$.
Finally, $(X^k, \Sa O^*, h)$ is definable in $\Sa M$.
\end{proof}

Lemma~\ref{lem;f} is a finitary form of Lemma~\ref{lem;f 0}.
It follows from Lemma~\ref{lem;f 0} by compactness, we leave the details to the reader.

\begin{lemma}\label{lem;f}
Fix structures $\Sa O\models T$, $\Sa M\models T^*$, and $k\ge 2$.
Suppose that $O\subseteq M^m$ and that $\Sa M$ trace defines $\Sa O$ via the inclusion $O\hookrightarrow M^m$.
Let $X$ be an $\Sa M$-definable set, $g$ be an $\Sa M$-definable function $X^k \to M^m$, and suppose that for every $n\ge 2$ and $\sigma\colon \nset^k \to O$ there are elements $(a_j^i : (i, j) \in \{1,\ldots,k\} \times \nset)$ of $X$ such that 
$$ g(a^1_{j_1}, \ldots, a^k_{j_k}) = \sigma(j_1, \ldots, j_k) \quad \text{for all\quad $j_1, \ldots, j_k \in \nset$.}$$
Then $T^*$ trace defines $D_k(T)$.
\end{lemma}

\subsection{Hilbert space}
\label{section;ips}
We consider any real Hilbert space to be a two-sorted structure of the form $(V,\R,\spcrngle)$ where we have vector addition on $V$, the ordered field structure on $\R$, scalar multiplication as a map $\R\times V\to V$, and the inner product $\spcrngle$ as a map $V\times V\to \R$.
We let $\ips$ be the theory of infinite-dimensional real Hilbert spaces.
Solovay, Arthan, and Harrison showed that $\ips$ is complete~\cite[Thm.~31]{solovay-arthan-harrison}.
We consider any complex Hilbert space to be a two-sorted structure of the form $(V,\C,\sigma,\spcrngle)$ where  we have vector addition on $V$, scalar multiplication as a map $\C\times V\to V$, the field structure on $\C$, $\sigma$ is complex conjugation $\C\to \C$, and $\spcrngle$ is the Hermitian form $V\times V\to \C$.
We let $\cips$ be the theory of infinite-dimensional complex Hilbert spaces.

\begin{proposition}
\label{prop;dobro}
$\ips$ and $\cips$ are both trace equivalent to $D_2(\rcf)$.
\end{proposition}


\begin{proof}
We first show that $\ips$ and $\cips$ are mutually interpretable and that $\cips$ is complete.
This boils down to basic theory of inner product spaces so we only give a sketch.

\medskip
Let $(V,\C,\sigma,\spcrngle)$ be an infinite-dimensional complex Hilbert space.
Note that $\R\subseteq\C$ is definable as it is the fixed field of $\sigma$.
Let $\langle v,w\rangle^*$ be the real part of $\langle v,w\rangle$ for any $v,w\in V$.
Then $(V,\R,\spcrngle^*)$ is a real Hilbert space.
Hence $\cips$ interprets $\ips$.

\medskip
We now show that $\ips$ interprets $\cips$.
Let $(V,\R,\spcrngle)$ be an infinite-dimensional real Hilbert space.
We consider the complexification of $(V,\R,\spcrngle)$.
We make $V\oplus V$ into a $\C$-vector space by declaring $\imag(v,w)=(-w,v)$ for all $v,w\in V$ and declare
$$\langle (v,w),(v',w')\rangle^* = \langle v,v'\rangle+\langle w,w'\rangle - \imag\langle v,w'\rangle +\imag\langle w,v'\rangle \quad\text{for all $(v,w),(v',w')\in V\oplus V$}. $$
Then $(V\oplus V,\C,\sigma,\spcrngle^*)$ is an infinite-dimensional complex Hilbert space which is definable in $(V,\R,\spcrngle)$.
Hence $\ips$ interprets $\cips$.
Furthermore, it is well-known that any complex Hilbert space is isomorphic to the complexification of a real Hilbert space of the same dimension.
Hence completeness of $\cips$ follows from completeness of $\ips$.

\medskip
We now show that $\ips$ is trace equivalent to $D_2(\rcf)$.
We first show that $\ips$ is trace definable in $D_2(\rcf)$.
Let $(V, \R, \spcrngle) \models \ips$.
It is enough to show that $(V, \R, \spcrngle)$ is $2$-trace definable in $\R$.
Let $v = (v_1,\ldots,v_n)$ be a tuple of vector variables and $r = (r_1, \ldots, r_m)$ be a tuple of scalar variables.
Let $\langle v \rangle$ be the tuple of inner products $\langle v_i, v_j \rangle$ for $i, j \in \{1, \ldots, n\}$.
Every formula in the variables $v, r$ in $(V, \R, \spcrngle)$ is equivalent to a formula of the form $\psi( \langle v \rangle, r)$ where $\psi(x_1, \ldots, x_{n^2 + m})$ is a formula in the language of ordered fields~\cite[Thm.~29]{solovay-arthan-harrison}.
Hence $\spcrngle$ and the identity $\R \to \R$ together witness $2$-trace definability of $(V, \R, \spcrngle)$ in $\R$.
It remains to show that $\ips$ trace defines $D_2(\rcf)$.
By Lemma~\ref{lem;f} it is enough to fix a function $\sigma \colon \nset^2 \to \R$ and produce vectors $v_1, w_1, \ldots, v_n, w_n \in V$ such that $\langle v_i, w_j \rangle = \sigma(i, j)$ for all $i, j \in \nset$.
Let $v_1, \ldots, v_n \in V$ be orthonormal and set $w_j = \sigma(1, j) v_1 + \cdots + \sigma(n, j) v_n$ for each $j = 1, \ldots, n$.
\end{proof}

\subsection{Vector spaces as two-sorted structures}
The following examples also involve the two-sorted theory of infinite-dimensional vector spaces over the models of a (complete) theory expanding the theory of fields.
We first recall this theory and its quantifier elimination.

\medskip
Let $L$ be a language expanding the language of fields and $T$ be an $L$-theory expanding the theory of fields.
Let $\Sa F$ range over models of $T$ and let $\F$ be the underlying field of $\Sa F$.
Given an $\F$-vector space $V$ we consider the two-sorted structure $(V, \Sa F)$ with sorts $V$ and $\F$ and:
\begin{enumerate}[leftmargin=*]
\item the full $L$-structure $\Sa F$ on $\F$, 
\item vector addition, subtraction, and $0$ on $V$, 
\item scalar multiplication as a map $\F\times V\to V$,
\item and $(n + 1)$-ary functions $\spn_{n, 1}, \ldots, \spn_{n, n} \colon V^{n + 1} \to \F$ for each $n\ge 1$ given by declaring $\spn_{n,i}(v_1,\ldots,v_{n},w)=0$ for each $i = 1,\ldots,n$ if either $v_1,\ldots,v_n$ are not independent or $w$ is not in the span of $v_1,\ldots,v_n$, and otherwise  the $\spn_{n, i}(v_1,\ldots,v_n,w)$ are the unique elements of $\F$ satisfying
$$w = \sum_{i=1}^n \spn_{n,i}(v_1,\ldots,v_n,w) v_i.$$
\end{enumerate}
We refer to the resulting language as $\lvec$.
Given an $L$-theory $T$ expanding the theory of fields we let $\vvec_T$ be the $\lvec$-theory of structures of the form $(V, \Sa F)$ for $\Sa F \models T$ and $V$ an infinite-dimensional $\F$-vector space.
Then $\vvec_T$ is complete by~\cite[Cor.~5.36]{aldaim}.
We refer to $\F$ as the {\bf scalar field} of $(V, \Sa F)$.
We also refer to a variable of the first sort as a {\bf vector variable} and a variable of the second sort as a {\bf scalar variable}.
Fact~\ref{fact:lspn} is due to Abd~Aldaim, Conant, and Terry~\cite[Cor.~5.36]{Kuzichev}.

\begin{fact}\label{fact:lspn}
If $T$ admits quantifier elimination then so does $\vvec_T$.
\end{fact}

If $\F$ is infinite then $\vvec_T$ has infinite dp-rank as there is a definable injection $\F^n \hookrightarrow V$ for every $n \ge 1$.
Hence $\vvec_T$ is not trace definable in $T$ when $T$ has finite dp-rank as finiteness of dp-rank is preserved under trace definability by~\cite[Prop.~4.1]{trace1}.
In particular $\vvec_T$ is not trace definable in $T$ when $T$ is the theory of algebraically, real, or $p$-adically closed fields.

\begin{proposition}
\label{prop;infn}
Let $L$ expand the language of rings, $\Sa F$ be an $L$-structure expanding a field $\F$, and $T = \Th(\Sa F)$.
\begin{enumerate}[leftmargin=*]
\item If $\F$ is finite then $\vvec_T$ is bi-interpretable with the usual one-sorted theory of $\F$-vector spaces.
\item If $\F$ has characteristic zero then $\vvec_T$ is locally trace equivalent to $T$.
\item If $\F$ has infinite imperfection degree then $\vvec_T$ is mutually interpretable with $T$.
\end{enumerate}
\end{proposition}


\begin{proof}
First note that (1) is immediate from the definition.
It is also clear that $\vvec_T$ always interprets $T$.
Suppose that $\F$ has infinite imperfection degree.
Let $p$ be the characteristic of $\F$.
Consider $\F$ to be a vector space over itself with scalar multiplication given by $(\lambda, v) \mapsto \lambda^p v$.
Equivalently, consider $\F$ to be a vector space over the subfield of $p$th powers, which is definably isomorphic to $\F$.
This gives an interpretation of $\vvec_T$ in $T$.

\medskip
It remains to show that $\vvec_T$ is locally trace definable in $T$ when $\F$ is characteristic zero.
First let us describe the motivation for our argument.
When $T = \acf_0$ one can show that $\vvec_T$ is locally trace definable in $T$ by noting that $\vvec_T$ is interpretable in the theory of an algebraically closed field of characteristic zero equipped with a unary relation defining a proper algebraically closed subfield, the latter theory is interpretable in $\dcf^1$, and $\dcf^1$ is locally trace definable in $\acf_0$.
Now, local trace definability of $(K, \der) \models \dcf^1$ in $K$ is witnessed by $\der$ and its iterates, so tracing through we see that local trace definability of $\vvec_T$ in $T$ is witnessed by $\der$ and its iterates in this case.
We generalize this idea.

\medskip
Suppose that $\F$ is characteristic zero.
After possibly Morleyizing we may suppose that $T$ admits quantifier elimination.
Then $\vvec_T$ admits quantifier elimination by Fact~\ref{fact:lspn}.
Now let $\Sa E$ be a non-trivial elementary extension of $\Sa F$ with underlying field $\E$.
Fix $t \in \E \setminus \F$ and let $V$ be $\F(t) \subseteq \E$, considered as an $\F$-vector space, so $(V, \Sa F) \models \vvec_T$.
It suffices to show that $\Sa E$ locally trace defines $(V, \Sa F)$.
Identify $\F(t)$ with the field of rational functions over $\F$ in one variable.
Let $\der\colon\F(t)\to\F(t)$ be the usual derivation and recall that $\der^{(n)}$ is the $n$-fold compositional iterate of $\der$.
In particular $\der^{(0)}$ is the identity on $\F(t)$.
Take each $\der^{(n)}$ to be a function $V \to \E$.
We show that $(\der^{(n)})_{n\in\N}$ witnesses local trace definability of $(V,\Sa F)$ in $\Sa E$.

\medskip
Let $v=(v_1,\ldots,v_n)$ be a tuple of vector variables and $c=(c_1,\ldots,c_m)$ be a tuple of scalar variables.
A vector, scalar term is an $\lvec$-term taking values in the sort $V,\F$, respectively.
Let $\Delta_k(v)$ be the  tuple $(\der^{(d)}(v_i))_{0\le d\le k - 1, 1\le i\le n}$ for each $k\ge 1$.
Given an $\Sa F$-definable function $h\colon \F^d\to\F$ we let $h^*\colon\E^d\to\E$ be the $\Sa E$-definable function defined by any formula defining $h$.
\begin{enumerate}[leftmargin=*]
\item A \textbf{$\Delta$-scalar function} is a function $V^n\times \F^m\to \F$ of the form $h^*(\Delta_k(v), c)$ for some $k\ge 1$ and $\Sa F$-definable function $h\colon \F^{kn}\times \F^m\to \F$.
\item A \textbf{$\Delta$-vector function} is a function $f \colon V^n\times\F^m\to V$ of the form
\[
f(v, c) = s_1(v, c) v_1 + \cdots + s_n(v, c) v_n.
\] 
for some $\Delta$-scalar functions $s_1, \ldots, s_n \colon V^n \times \F^m \to \F$.

\end{enumerate}
Finally, a $\Delta$-function is a function that is either a $\Delta$-scalar function or a $\Delta$-vector function.
\begin{Claim*}
Any $\lvec$-term in the variables $v,c$ defines a $\Delta$-function.
\end{Claim*}

\begin{claimproof}
Note that each $v_i$ trivially defines a $\Delta$-vector function and each $c_i$ trivially defines a $\Delta$-scalar function.
Hence it is enough to show that $\Delta$-functions are closed under the $\lvec$-term building operations.
We first make some easy observations.
\begin{enumerate}[leftmargin=*, label=(\alph*)]
\item If $s_1(v,c), \ldots, s_d(v,c)$ are $\Delta$-scalar functions and $h\colon \F^d \to \F$ is $\Sa F$-definable then $$h(s_1(v,c) ,\ldots, s_d(v,c))$$ is a $\Delta$-scalar function.
Hence $\Delta$-scalar functions are closed under sums and products.
\item $\Delta$-vector functions are closed under finite sums (combine like terms).
\item If $s(v,c)$ is a $\Delta$-scalar function and $f(v,c)$ is a $\Delta$-vector then $s(v,c)f(v,c)$ is also a $\Delta$-vector function (write $f(v,c)$ out as in the definition and distribute $s(v,c)$).
\end{enumerate}
It remains to show that $\spn_{d, i}(t_1,\ldots,t_{d+1})$ is a $\Delta$-scalar function for fixed $\Delta$-vector functions $t_1(v,c), \ldots, t_{d + 1}(v,c)$.
We first show that $\spn_{d,i}(w_1,\ldots,w_d, u)$ is a $\Delta$-scalar function in the vector variables $w_1,\ldots,w_d,u$ for fixed $d \ge 1$ and $i = 1, \ldots, d$.
We treat the case $i = 1$, the general case follows in the same way.
Set $w=(w_1,\ldots,w_d)$.
We let $\wron(w)$ be the Wronskian of $w$.
This is the following matrix over $\F(t)$.
$$\wron(w) =\begin{pmatrix}
    w_1 & w_2 &\cdots& w_d\\
    \der(w_1)&\der(w_2)&\cdots& \der(w_d)\\
    \vdots&\vdots&\ddots&\vdots
    \\
    \der^{(d-1)}(w_1)&\der^{(d-1)}(w_2)&\cdots& \der^{(d-1)}(w_d)
\end{pmatrix}$$
Recall that the $w_i$ are $\F$-linearly independent if and only if $\wron(w)$ is invertible~\cite[Lemma~4.1.13]{trans}.
(This is the point at which we use characteristic zero.)
Now let $u=\lambda_1 w_1+\cdots+\lambda_d w_d$ for $\lambda_1,\ldots,\lambda_d\in\F$.
Then we have
$$\der^{(i)}(u) = \lambda_1 \der^{(i)}(w_{1}) +\cdots+\lambda_d \der^{(i)}(w_{d})\quad\text{for each $i = 1,\ldots,d-1$ }.$$
Hence
$$\begin{pmatrix}
    w_1 & w_2 &\cdots& w_d\\
    \der(w_1)&\der(w_2)&\cdots& \der(w_d)\\
    \vdots&\vdots&\ddots&\vdots
    \\
    \der^{(d-1)}(w_1)&\der^{(d-1)}(w_2)&\cdots& \der^{(d-1)}(w_d)
\end{pmatrix} \begin{pmatrix}
    \lambda_1 \\\lambda_2\\\vdots\\\lambda_d\end{pmatrix}
=\begin{pmatrix}u\\\der(u)\\\vdots\\\der^{(d-1)}(u)\end{pmatrix}.$$
The vector on the right hand side of this equation is $\Delta_d(w)$.
Thus if $w_1,\ldots,w_d$ are linearly independent and $u$ is in the span of $w_1,\ldots,w_d$ then $\wron(w_1,\ldots,w_d)^{-1}\Delta_d(u)$ is the coordinate vector of $u$ with respect to $w_1,\ldots,w_d$.
Hence  we have:
\begin{enumerate}[leftmargin=*]
\item $\spn_{d, 1}(w,u)=0$ when either $\det\wron(w)=0$ or $\det\wron(w)\ne 0\ne\det\wron(w,u)$,
\item and otherwise  $\spn_{d,1}(w,u)$ is the first coordinate of $\wron(w)^{-1}\Delta_d(u)$.
\end{enumerate}
It follows that $\spn_{d,1}(w_1,\ldots,w_d,u)$ is a $\Delta$-scalar function in the vector variables $w_1,\ldots,w_d,u$.
Hence it is enough to fix a $\Delta$-scalar function $g(w_1,\ldots,w_d)$ in the vector variables $w_1,\ldots,w_d$ and show that $g(t_1,\ldots,t_d)$ is a $\Delta$-scalar function for $\Delta$-vector functions $t_1(v,c),\ldots,t_d(v,c)$.
Fix $k \ge 1$, $\Sa F$-definable $f \colon \F^k \to \F$, $i_1, \ldots, i_k \in \{1,\ldots,d\}$, and $j_1, \ldots, j_k \in \N$ such that
\[
g(w_1,\ldots,w_d) = f^*\left(\der^{(j_1)}(w_{i_1}),\ldots,\der^{(j_k)}(w_{i_k})\right).
\]
Fix $\Delta$-scalar functions $s_{i,1}, \ldots, s_{i,n} \colon V^n \times \F^m\to \F$ for each $i = 1,\ldots,d$
such that
\[
t_i(v, c) = s_{i,1}(v, c) v_1 + \cdots + s_{i,n}(v, c) v_n.
\]
For each $j\in\N$ we have
\[
\der^{(j)}(t_i(v, c)) = s_{i,1}(v, c) \der^{(j)}(v_1) + \cdots + s_{i,n}(v, c) \der^{(j)}(v_n).\]
Here we have used the fact that each $s_{i,\ell}(v, c)$ takes values in $\F$  to pass $\der^{(j)}$ through.
So each $\der^{(j)}(t_i(v, c))$ is a sum of products of $\Delta$-scalar functions and is hence a $\Delta$-scalar function.
Now
\[
g(t_1(v, c), \ldots, t_d(v, c)) = f^*\left( \der^{(j_1)}(t_{i_1}(v, c)), \ldots, \der^{(j_k)}(t_{i_k}(v, c)) \right).
\]
Hence $g(t_1(v,c),\ldots,t_d(v,c))$ is a $\Delta$-scalar function by (a) above.
\end{claimproof}
We finally show that the $\der^{(i)}$ witness local trace definability of $(V, \Sa F)$ in $\Sa E$.
Let $X$ be a $(V, \Sa F)$-definable subset of $V^n \times \F^m$.
An application of quantifier elimination for $\vvec_T$ shows that $X$ is a boolean combination of sets of  the following forms:
\begin{enumerate}[leftmargin=*]
\item $\{(a,b)\in V^n \times \F^m : (t_1(a, b), \ldots, t_d(a, b)) \in Y\}$ where $Y$ is an $\Sa F$-definable subset of $\F^d$ and $t_1, \ldots, t_d$ are $\lvec$-terms taking values in $\F$.
\item $\{(a,b) \in V^n \times \F^m : t(a, b) = \eta\}$ where $t$ is an $\lvec$-term taking values in $V$ and $\eta \in V$.
\end{enumerate}
Suppose that $X$ is as in (1).
By the claim each $t_i$ is a $\Delta$-scalar function.
Fix $k$ and $\Sa F$-definable functions $h_1, \ldots, h_d \colon \F^{kn} \times \F^m \to \F$ such that $t_i(a, b) = h^*_i(\Delta_k(a), b)$ for each $i = 1,\ldots, d$ and $(a, b) \in V^n \times \F^m$.
Let $Y^*$ be the subset of $\E^d$ defined by any formula defining $Y$.
So for any $(a, b) \in V^n \times \F^m$ we have $(a, b) \in X$ if and only if $(h^*_1(\Delta_k(a), b), \ldots, h^*_d(\Delta_k(a), b)) \in Y^*$.
Finally, let $Z$ be the set of $(\beta, \gamma) \in \E^m \times \E^{kn}$ such that $(h^*_1(\gamma, \beta), \ldots, h^*_d(\gamma, \beta)) \in Y^*$.
Then $Z$ is definable in $\Sa E$ and we have $(a, b) \in X$ if and only if $(b, \Delta_k(a)) \in Z$ for any $(a, b) \in V^n \times \F^m$.

\medskip
Now suppose that $X$ is as in (2).
By the claim $t$ is a $\Delta$-vector function.
Hence there is $k\ge 1$ and $\Sa F$-definable functions $h_1,\ldots,h_n\colon \F^{kn}\times \F^m\to \F$ such that \[t(a, b) = h^*_1(\Delta_k(a), b) a_1 + \cdots  +h^*_n(\Delta_k(a), b) a_n\]
and  $h^*_1(\Delta_k(a),b), \ldots, h^*_n(\Delta(a), b) \in \F$ for all $a=(a_1,\ldots,a_n) \in V^n$ and $b \in \F^m$.
Now let $Z$ be the set of $(\alpha_1, \ldots, \alpha_n, \beta, \gamma) \in \E^n \times \E^m \times \E^{kn}$ such that
\[
h^*_1(\gamma, \beta) \alpha_1 + \cdots  + h^*_n(\gamma, \beta) \alpha_n = \eta.
\]
Now $Z$ is definable in $\Sa E$ and we have $(a, b) \in X$ if and only if $(a, b, \Delta_k(a)) \in Z$ for any $(a, b) \in V^n \times \F^n$.
\end{proof}

Suppose $\F$ is characteristic zero.
We have shown that local trace definability of $\vvec_T$ in $T$ is witnessed by countably many functions.
Hence $\vvec_T$ is trace definable in $D^{\aleph_0}(T)$ by Theorem~\ref{thm;dock}(3).
So it is natural to ask if $\vvec_T$ is trace equivalent to $D^{\aleph_0}(T)$.
We show that this fails in many cases of interest.

\begin{proposition}\label{prop:not T}
Let $T$ be as in the previous proposition.
Suppose that $T$ is strongly dependent and geometric.
Then $\vvec_T$ does not trace define $D^{\aleph_0}(T)$.
\end{proposition}

\begin{proof}
Strong dependence is preserved under trace definability~\cite[Prop.~4.1]{trace1}.
It is easy to see that $D^{\aleph_0}(\triv)$ is not strongly dependent.
Hence it suffices to show that $\vvec_T$ is strongly dependent.
It is easy to see that $\vvec_T$ is interpretable in the theory of lovely pairs of models of $T$, and the theory of lovely pairs of models of $T$ is strongly dependent by \cite[Cor.~3.3]{gdensepairs}.
\end{proof}

\subsection{Nilpotent Lie algebras}\label{section;vecT}
Let  $\Gamma$ be either a group or a Lie algebra,  and  $\spcbrckt$ be either the group commutator or the Lie bracket on $\Gamma$, respectively.
Recall that $\Gamma$ is said to be nil-$k$ if it is nilpotent of class at most $k$.
A \textbf{Lazard sequence} is a finite sequence $P_0\supseteq\cdots\supseteq P_{k}$ of substructures of $\Gamma$ such that $P_0=\Gamma$, $P_{k}$ is the trivial substructure, and $[P_i,P_j]\subseteq P_{\min\{i+j,k\}}$ for all $i,j$.
Hence if $\Gamma$ is nil-$k$ then the lower central series of $\Gamma$ is a Lazard sequence and $\Gamma$ admits a Lazard sequence of length $k$ if and only if $\Gamma$ is nil-$k$.
A \textbf{Lazard group (Lie algebra)} is a group (Lie algebra) equipped with a Lazard sequence.

\medskip
Let $L$ expand the language of rings, $\Sa F$ be an $L$-structure expanding a field $\F$, and  $T=\Th(\Sa F)$. 
Given $k\ge 2$ let $\leftindex _{\_}\nil^k_T$ be the theory of structures of the form $(V, \Sa F, \spcbrckt,P_0,\ldots,P_k)$ where $(V, \Sa F) \models \vvec_T$, $\spcbrckt$ is a Lie bracket on $\Sa V$, and $P_0,\ldots,P_k$ is a Lazard sequence of the associated Lie algebra.
Hence the underlying Lie algebra of any model of $\leftindex _{\_}\nil^k_T$ is nil-$k$.
Furthermore if $\E$ is the underlying field of some $\Sa E\models T$ then any infinite-dimensional nil-$k$ Lie algebra over $\E$ expands to a model of $\leftindex _{\_}\nil^k_T$ in an obvious way.
Fact~\ref{fact;vecT} is due to d'Elb\'{e}e, M\"uller, Ramsey and Siniora~\cite[Lemmas~3.1 and 3.2]{delbée2024twosortedtheorynilpotentlie}.

\begin{fact}\label{fact;vecT}
Let $L$ and $T$ be as above and suppose that $T$ admits quantifier elimination.
Then $\leftindex _{\_}\nil^k_T$ has a model companion $\nil^k_T$ and $\nil^k_T$ is complete and admits quantifier elimination.
\end{fact}

We now relate $\nil^k_T$ to $D_k(\vvec_T)$.

\begin{proposition}\label{prop;vecT}
Let $L, \Sa F, \F, T$ be as above.
Suppose that $T$ admits quantifier elimination and fix $k\ge 2$.
Then $\nil^k_T$ is trace equivalent to $D_k(\vvec_T)$.
Hence if $\F$ is characteristic zero  then $\nil^k_T$ is locally trace equivalent to $D_k(T)$ and if $\F$ has infinite imperfection degree then $\nil^k_T$ is trace equivalent to $D_k(T)$.
\end{proposition}

Of course we can drop the assumption that $T$ admits quantifier elimination by Morleyizing.

\medskip
We will need to use some Lie theory.
We let $\Sa A \oplus \Sa B$ be the direct sum of Lie $\F$-algebras $\Sa A$ and $\Sa B$.
This is the Lie $\F$-algebra whose underlying vector space is the direct sum of the underlying vector spaces of $\Sa A$ and $\Sa B$ and with Lie bracket given by
$$[a + b, a^* + b^*] = [a,a^*] + [b,b^*]\quad\text{for all  $a, a^* \in A$ and $b, b^* \in B$.}$$
Hence in particular $[a, b]=0$ for any $a \in A$ and $b \in B$.
A \textbf{Lie monomial} is an element of the smallest collection of terms containing all vector variables and closed under applying $\spcbrckt$.
We let $[x_1,\ldots,x_n]$ be the Lie monomial given inductively by $[x_1,\ldots,x_n] = [[x_1,\ldots,x_{n-1}],x_n]$ for all $n\ge 3$.
We also use some basic facts about free nilpotent Lie algebras which can be found in \cite[\S~4.2]{2024modeltheoreticpropertiesnilpotentgroups}.
We now prove Proposition~\ref{prop;vecT}.

\begin{proof}
The second claim follows from the first claim by Proposition~\ref{prop;infn} and Corollary~\ref{cor;dock}.
Hence it is enough to show that $\nil^k_T$ is trace equivalent to $D_k(\vvec_T)$.

\medskip
We first show that $\nil^k_T$ trace defines $D_k(\vvec_T)$.
Fix $\Sa V = (V, \Sa F) \models \vvec_T$.
We consider $\Sa V$ to be a model of $\leftindex _{\_}\nil^k_T$ by declaring $[a,b]=0$ for all $a,b\in V$, $P_0 = V$, and $P_1 = \ldots = P_k = \{0\}$.
Let $\Sa V^*$ be the structure induced on $V$ by $\Sa V$.
Now $\Sa V^*$ interprets $\Sa V$, so it is enough to show that $\nil^k_T$ trace defines $D_k(\Th(\Sa V^*))$.
Fix a set $P$ and a function $f \colon P^k \to V$.
We construct a nil-$k$ Lie $\F$-algebra $\Sa B$  extending $\Sa V$ and elements $(a^i_p : (i, p) \in \{1,\ldots,k\} \times P)$ of $\Sa B$ such that
\[
[a^1_{p_1},\ldots,a^k_{p_k}] = f(p_1, \ldots, p_k) \quad \text{for all  } p_1,\ldots,p_k \in P.
\]
We then consider $\Sa B$ to be a model of $\leftindex _{\_}\nil^k_T$ in the natural way and let $\Sa W$ be a model of $\nil^k_T$ extending $\Sa B$.
Let $(W, \Sa E)$ be the $\lvec$-reduct of $\Sa W$.
Now $\vvec_T$ is model complete by Fact~\ref{fact:lspn}, so $(V, \Sa F)$ is an elementary substructure of $(W, \Sa E)$.
Hence an application of Lemma~\ref{lem;f 0} will show that $\nil^k_T$ trace defines $D_k(\Th(\Sa V^*))$.

\medskip
Let $\Sa C$ be the free nil-$k$ Lie $\F$-algebra with generators $x^i_p$ for $(i, p) \in \{1,\ldots,k\} \times P$.

\begin{Claim*}
$\{ [x^1_{p_1}, \ldots, x^k_{p_k}] : p_1, \ldots, p_k \in P\}$ is a linearly independent subset of $\Sa C$.
\end{Claim*}

\begin{claimproof}
We apply the theory of Hall bases~\cite[\S~4.2.2]{2024modeltheoreticpropertiesnilpotentgroups}.
Fix an arbitrary linear order on $P$, equip $\{1, \ldots, k\} \times P$ with the resulting lexicographic order, and order the $x^i_p$ accordingly.
Fix $p_1, \ldots, p_k$.
So we have $x^1_{p_1} < x^2_{p_2} < \ldots < x^k_{p_k}$.
It follows from the definition that 
\[
[\ldots[[[x^2_{p_2}, x^1_{p_1}], x^3_{p_3}], x^4_{p_4}], \ldots, x^k_{p_k}] = [x^2_{p_2}, x^1_{p_1}, x^3_{p_3}, x^4_{p_4}, \ldots, x^k_{p_k}]
\]
is a Hall basis element of $\Sa B$ with respect to this order.
Hence these elements are linearly independent.
Now we have $[x^1_{p_1}, x^2_{p_2}] = - [x^2_{p_2}, x^1_{p_1}]$, so by bilinearity of the bracket we have
\[
[x^2_{p_2}, x^1_{p_1}, x^3_{p_3}, x^4_{p_4}, \ldots, x^k_{p_k}] = - [x^1_{p_1}, x^2_{p_2}, x^3_{p_3}, x^4_{p_4}, \ldots, x^k_{p_k}].
\]
The claim follows.
\end{claimproof}

We consider the Lie algebra $V \oplus \Sa C$.
Let
\[
S = \{ f(p_1, \ldots, p_k) - [x^1_{p_1}, \ldots, x^k_{p_k}] : p_1,\ldots,p_k\in P \} \subseteq V \oplus \Sa C
\]
and let $I$ be the subspace of $V \oplus \Sa C$ spanned by $S$.
We show that $I$ is a Lie ideal.
We need to show that $[v, w] \in I$ for any $v \in I$ and $w \in V \oplus \Sa C$.
It suffices to show that $[v, w] = 0$ for any $v \in S$ and $w \in V \oplus \Sa C$.
For any $c \in \Sa C$, $v \in V$, and $p_1,\ldots,p_k \in P$ we have
\begin{align*}
\left[f(p_1, \ldots, p_k) - [x^1_{p_1}, \ldots, x^k_{p_k}], v + c \right] \quad &=\quad
[f(p_1, \ldots, p_k), v] + [f(p_1,\ldots,p_k), c] \\ &\quad - [x^1_{p_1}, \ldots, x^k_{p_k}, v] - [x^1_{p_1}, \ldots, x^k_{p_k}, c] \\
&= 0 + 0 - 0 - 0  = 0.
\end{align*}
Here the second and third terms vanish by definition of the direct sum of Lie algebras, the first term vanishes as the bracket on $V$ is trivial, and the last term vanishes as $\Sa C$ is nil-$k$.
Hence $I$ is a Lie ideal.
Let $\Sa B$ be the Lie algebra $(V \oplus \Sa C)/I$.
It follows from the claim that $\{ [x^1_{p_1}, \ldots, x^k_{p_k}] : p_1, \ldots, p_k \in P\}$ is linearly independent over $V$.
Hence $I\cap V= \{0\}$, so we may identify $V$ with its image under the quotient map $V \oplus \Sa C \to B$.
Furthermore identify each $x^i_p$ with its image under the quotient map.
Then $\Sa B$ satisfies 
$$[x^1_{p_1},\ldots,x^k_{p_k}] = f(p_1,\ldots,p_k) \quad\text{for all } p_1,\ldots,p_k \in P.$$
Finally note that $\Sa B$ is nil-$k$ as $V \oplus \Sa C$ is nil-$k$.

\medskip
It remains to show that $D_k(\vvec_T)$ trace defines $\nil^k_T$.
It is enough to show that $\nil^k_T$ is $k$-trace definable in $\vvec_T$.
Fix $(V, \Sa F, \spcbrckt,P_0,\ldots,P_k) \models \nil^k_T$.
Let $\Sa W = (V, \Sa F, P_0,\ldots,P_k)$.
This is a model of $\vvec_T$ equipped with a strictly descending chain of subspaces such that each quotient $P_i/P_{i + 1}$ is infinite-dimensional.
Observe that $\Sa W$ is isomorphic to the structure $(V^k, \Sa F, Q_0, \ldots, Q_k )$ where $Q_0 = V^k$ and
\[
Q_i = \{(v_0, \ldots, v_{k - 1}) \in V^k : v_{k - i} = v_{k - i + 1} = \ldots = v_{k - 1} = 0\} \quad \text{for each $i = 1, \ldots, k$.}
\]
It follows that $\Sa W$ is interpretable in $\vvec_T$.
We show that $(\Sa W, \spcbrckt)$ is $k$-trace definable in $\Sa W$.
 
\medskip
As $(\Sa W,\spcbrckt)$ is nil-$k$ the collection of Lie monomials gives only a finite collection of functions, each of arity at most $k$.
Hence it is enough to show that any formula in $(\Sa W, \spcbrckt)$ in vector variables $x_1,\ldots,x_n$ and scalar variables $y_1, \ldots, y_m$ is equivalent to a formula of the form
$$ \vartheta(x_1, \ldots, x_n, y_1, \ldots, y_m, g_1(x_{i_{1,1}},\ldots,x_{1,k}),\ldots,g_d(x_{i_{d,1}},\ldots,x_{d,k}))$$
for some formula $\vartheta(z_1,\ldots,z_{n + m + d})$ in $\Sa W$, Lie monomials $g_1,\ldots,g_d$, and indices $i_{j, l}$ from $\nset$.
This follows by the quantifier elimination given in \cite[Lemma~3.2]{delbée2024twosortedtheorynilpotentlie} together with the description of quantifier-free formulas given in \cite[Lemma~4.2]{delbée2024twosortedtheorynilpotentlie}.
(The description of terms is only written to cover the case when $T$ is the theory of an algebraically closed field but immediately generalizes.)
\end{proof}

We now consider the case of Proposition~\ref{prop;vecT} when $T = \Th(\F_p)$.
Recall that $\vvec_p$ is the (one-sorted) theory of $\F_p$-vector spaces and $\vvec_p$ is bi-interpretable with $\vvec_T$ in this case.

\begin{proposition}\label{prop;f f}
Fix $k\ge 2$ and a prime $p$.
Then $D_k(\vvec_p)$ is trace equivalent to the theory of the \Fraisse limit of the class of finite nil-$k$ Lazard Lie algebras over $\F_p$.
If $k < p$ then $D_k(\vvec_p)$ is trace equivalent to the theory of the \Fraisse limit of the class of finite nil-$k$ Lazard groups of exponent $p$.
\end{proposition}

Proposition~\ref{prop;f f} follows from Proposition~\ref{prop;vecT} and Fact~\ref{fact:laz} below.
Suppose $k < p$.
We apply the Lazard correspondence between finite nil-$k$ groups of  exponent $p$ and finite nil-$k$ Lie algebras over $\F_p$~\cite{Lazard}.
The Lazard correspondent $\Gamma_\mathrm{Laz}$  of a finite exponent $p$ nil-$k$ group $\Gamma$ has the same domain as $\Gamma$ and  $\Gamma_\mathrm{Laz}$ is interdefinable with $\Gamma$~\cite[\S~2.3]{2024modeltheoreticpropertiesnilpotentgroups}.
Fact~\ref{fact:laz} is due to d'Elb\'{e}e, M\"uller, Ramsey and Siniora~\cite[Thm.~4.37 and Cor.~4.39]{2024modeltheoreticpropertiesnilpotentgroups}.

\begin{fact}
\label{fact:laz}
Fix $k\ge 2$ and a prime $p$.
\begin{enumerate}
[leftmargin=*]
\item Finite nil-$k$ Lazard Lie algebras over $\F_p$ form a \Fraisse class.
Let $\Sa A_{p,k}$ be the \Fraisse limit.
Then $\Sa A_{p,k}$ is interdefinable with its underlying Lie algebra.
\item  Suppose $k < p$.
Then finite nil-$k$ Lazard groups of exponent $p$ form a \Fraisse class.
Let $\Sa G_{p,k}$ be the \Fraisse limit.
Then $\Sa G_{p,k}$ is interdefinable with its underlying group and furthermore the underlying group is the Lazard correspondent of the underlying Lie algebra of $\Sa A_{p,k}$.
It follows that $\Sa G_{p,k}$ is bidefinable with $\Sa A_{p,k}$.
\end{enumerate}
\end{fact}

We finally describe another theory that is trace equivalent to $D_2(\vvec_p)$.
Let $\mathrm{Bil}_p$ be the class of two-sorted structures of the form $(V,W,\beta)$ where $V$ and $W$ are finite $\F_p$-vector spaces and $\beta$ is an alternating bilinear form $V\times V \to W$.
Baudisch showed that if $p > 2$ then $\mathrm{Bil}_p$ is a \Fraisse class whose limit is bidefinable with $\Sa A_{p,2}$~\cite{BAUDISCH_nilpotent}.
Therefore Proposition~\ref{prop;p2} is a consequence of Proposition~\ref{prop;f f}.

\begin{proposition}
\label{prop;p2}
If $p$ is an odd prime then $D_2(\vvec_p)$ is trace equivalent to the theory of the \Fraisse limit of $\mathrm{Bil}_p$.
\end{proposition}

This can also be proven directly along the lines of Proposition~\ref{prop;dobro}.




\subsection{Multilinear forms}\label{section;alt}
Fix $k\ge 2$.
Let $L$ be the language of rings, $\F$ be a field, and $T=\Th(\F)$.
We consider a theory of vector spaces over models of $T$ equipped with alternating $k$-linear forms.
We let $\lalt$ be the expansion of $\lvec$ by a $k$-ary function $\beta$ and an $n$-ary relation $R_\varphi$ for every $n$-ary formula in the language of rings.
Given $\E \models T$, an $\E$-vector space $V$, and an alternating $k$-linear form $V^k \to \E$, we produce an $\lalt$-structure by considering $(V, \E)$ to be an $\lvec$-structure as described above, letting $\beta$ be the form, and letting each $n$-ary $R_\varphi$ define the subset of $\E^n$ defined by $\varphi$.
We will simply denote this structure by $(V, \E, \beta)$ and let $\mathrm{Alt}_{\F,k}$ be the theory of such structures.

\medskip
Chernikov and Hempel define a notion of non-degeneracy for multilinear forms and show that the $\lalt$-theory $\mathrm{Alt}^*_{T,k}$ of models $(V, \E, \beta)$ of $\mathrm{Alt}_{\F,k}$ with $V$ infinite-dimensional and $\beta$ non-degenerate is complete and admits quantifier elimination~\cite[Thm.~2.19]{chernikov2025ndependentgroupsfieldsiii}.
When $k = 2$ their notion of non-degeneracy agrees with the usual notion.
Furthermore every model of $\mathrm{Alt}_{T,k}$ embeds into a model of $\mathrm{Alt}^*_{T,k}$ with the same scalar field~\cite[Lemma~2.4]{chernikov2025ndependentgroupsfieldsiii}.
In particular $\mathrm{Alt}^*_{T,k}$ is the model companion of $\mathrm{Alt}_{T, k}$ and there is a model of $\mathrm{Alt}^*_{T, k}$ with scalar field $\F$.

\begin{proposition}\label{prop;ccs}
Let $T$ be the theory of a field $\F$.
Then $\mathrm{Alt}^*_{T,k}$ is trace definable in $D_k(\vvec_T)$ and trace defines $D_k(T)$.
If $\F$ is characteristic zero then $\mathrm{Alt}^*_{T,k}$ is locally  trace equivalent to $D_k(T)$ and if $\F$ has infinite imperfection degree then $\mathrm{Alt}^*_{T, k}$ is trace equivalent to $D_k(T)$.
\end{proposition}

\begin{proof}
Fix $(V, \F, \beta) \models \mathrm{Alt}^*_{T,k}$.
We first show that $(V, \F, \beta)$ is trace definable in $D_k(\vvec_T)$.
It is enough to show that $(V, \F, \beta)$ is $k$-trace definable in $(V, \F) \models \vvec_T$.
As noted above $(V, \F, \beta)$ admits quantifier elimination in $\lalt$ and by \cite[Lemma~5.9]{aldaim} any term in $(V, \F, \beta)$ in vector variables $v_1,\ldots,v_n$ and scalar variables $c_1, \ldots, c_m$ is equivalent to a term of the form 
$$t\left(v_1,\ldots,v_n, c_1, \ldots, c_m, \beta(v_{i_{1,1}},\ldots,v_{i_{1,k}}), \ldots, \beta(v_{i_{d,1}},\ldots,v_{i_{d,k}})\right)$$
for some $\lvec$-term $t$ and indices $i_{j,l}$ from $\{1,\ldots,n\}$.
Hence $\beta$ and the identity $\F \to \F$ together witness $k$-trace definability of $(V, \F, \beta)$ in $(V, \F)$.

\medskip
We now show that $\altT$ trace defines $D_k(T)$.
We apply Lemma~\ref{lem;f 0}.
Fix a set $P$ and a function $f \colon P^k \to \F$.
Let $V$ be the $\F$-vector space with basis $(x^i_p : (i, p) \in \{1,\ldots,k\} \times P)$.
Fix an arbitrary linear order on $P$ and order the $x^i_p$ according to the resulting lexicographic order on $\{1,\ldots,k\} \times P$.
Now let $W$ be the $k$th exterior power of $V$.
Then each $x^1_{p_1} \wedge \cdots \wedge x^k_{p_k}$ is an element of the standard ordered basis of $W$ and so these elements are linearly independent.
Hence there is a linear map $\gamma \colon W \to \F$ such that 
\[
\gamma(x^1_{p_1} \wedge \cdots \wedge x^k_{p_k}) = f(p_1, \ldots, p_k) \quad \text{for all $p_1, \ldots, p_k \in P$.}
\]
Now let $\beta \colon V^k \to \F$ be given by $\beta(v_1,\ldots,v_k) = \gamma(v_1\wedge\cdots\wedge v_k)$ for all $v_1,\ldots,v_k\in V$.
Then $\beta$ is a $k$-linear alternating form satisfying 
\[
\beta(x^1_{p_1}, \ldots, x^k_{p_k}) = f(p_1, \ldots, p_k) \quad \text{for all $p_1,\ldots,p_k \in P.$}
\]
Finally, embed $(V, \F, \beta)$ into a model of $\altT$ with the same scalar field.
\end{proof}

\subsection{Uryshon spaces}\label{section:ury}
Let $\Sa R = \rgroop$ be an ordered abelian group and $X$ be a set.
An $\Sa R$-valued metric $d$ on $X$ is a function $X^2 \to R$ satisfying the usual metric space axioms: $d(x, y) = 0$ iff $x = y$, $d$ is symmetric, and $d(x, z) \trianglelefteq d(x,y) + d(y, z)$.
We say that an $\Sa R$-valued metric $d$ on $X$ is universal if every finite $\Sa R$-valued metric space isometrically embeds into $(X, d)$ and is homogeneous if every isometry between finite subsets of $X$ extends to an isometry $X \to X$.
We say that an $\Sa R$-valued metric is {\bf Uryshon} if whenever  $(Y, d^*)$ is a finite $\Sa R$-valued metric space, then any isometric embedding of a finite subset of $Y$ into $X$ extends to an isometric embedding of $Y$ into $X$.
The classical Uryshon space is the unique up to isometry separable complete $\rgoup$-valued Uryshon metric space.

\begin{fact}\label{fact:gabe}
If $\Sa R$ is countable then there is a unique up to isometry universal homogeneous $\Sa R$-valued metric space $\Sa U_\Sa R$.
Any universal homogeneous $\Sa R$-valued metric space is Uryshon.
\end{fact}

The second claim of Fact~\ref{fact:gabe} is immediate from the definitions.
The first is~\cite[Thm.~5.5]{Conant_distance}.
If $\Sa R = (\Q; +, <)$ then $\Sa U_\Sa R$ is the rational Uryshon space and the completion of $\Sa U_\Sa R$ is the classical Uryshon space.
When $\Sa R = (\Z; +, <)$, $\Sa U_\Sa R$ is called the integral Uryshon space.

\medskip
Now let $L$ be a language expanding the language of ordered abelian groups and $T$ be an $L$-theory expanding the theory of ordered abelian groups.
For our purposes a $T$-valued metric space is a two sorted structure $(X, \Sa M, d)$ where $\Sa M \models T$, $X$ is a set, and $d$ is an $\Sa M$-valued metric on $X$.
We refer to $X$ as the point sort and $\Sa M$ as the distance sort.
We say that a $T$-valued metric space is finite when the point sort is finite.

\medskip
It is clear from the definition that $T$-valued Uryshon spaces form an elementary class; we let $\tu$ be their theory.
Let $T_0$ be the reduct of $T$ to the language of ordered abelian groups.
Then $T_0$ has a countable model, so by Fact~\ref{fact:gabe} there is an Uryshon $T_0$-valued metric space.
It follows by the Robinson joint consistency theorem that $\tu$ is consistent, i.e. there are Uryshon $T$-valued metric spaces.

\begin{proposition}\label{prop:ury}
Let $T$ be as above.
Then $\tu$ is trace equivalent to $D_2(T)$.
\end{proposition}

Recall that $\Sa R$ is said to be non-singular if $pR$ has finite index in $R$ for every prime $p$.
All non-singular ordered abelian groups are trace equivalent~\cite[Prop.~4.19]{trace2}.
It follows that if $\Sa R$ is non-singular and countable then $\Sa U_\Sa R$ is trace equivalent to the classical Uryshon space.
In particular the integral Uryshon space is trace equivalent to the classical Uryshon space.
Combining with Proposition~\ref{prop;dobro} we also see that the classical Uryshon space, when considered as an $\rfield$-valued metric space, is trace equivalent to infinite-dimensional Hilbert space.
Proposition~\ref{prop:ury} requires several steps.
We first show that $\tu$ is the relative model companion of the theory of $T$-valued metric spaces.

\begin{lemma}\label{lem:ury model comp}
Let $T$ be as above and suppose that $T$ admits quantifier elimination.
Then $\tu$ is complete, is the model companion of the theory of $T$-valued metric spaces, and admits quantifier elimination.
\end{lemma}


\begin{proof}
We first need to show that any $T$-valued metric space embeds into an Uryshon $T$-valued metric space.
It suffices to treat the case when $T = \Th(\Sa R)$.
We may suppose that $\Sa R$ is countable.
Then $\Sa U_\Sa R$ exists and any countable $\Sa R$-valued metric space embeds into $\Sa U_\Sa R$.
It follows by an obvious compactness argument that any $T$-valued metric spaces embeds into an Uryshon $T$-valued metric space.

\medskip
Fix $(X, \Sa M, d) \models \tu$ and a $T$-valued metric space $(Y, \Sa N, e)$ extending $(X, \Sa M, d)$.
We show that $(X, \Sa M, d)$ is existentially closed in $(Y, \Sa N, e)$.
It follows that $\tu$ is the model companion of the theory of $T$-valued metric spaces.
Let $(X^*, \Sa M^*, d^*)$ be a $\max(|X|, |N|)^+$-saturated elementary extension of $(X, \Sa M, d)$.
It is enough to show that there is an embedding $(Y, \Sa N, e) \hookrightarrow (X^*, \Sa M^*, d^*)$ which fixes both $X$ and $M$ pointwise.
As $T$ admits quantifier elimination $\Sa N$ is an elementary extension of $\Sa M$, so by saturation we may suppose that $\Sa N$ is an elementary submodel of $\Sa M^*$.
Considering $e$ as a function $Y^2 \to M^*$ we have a $T$-valued metric space $(Y, \Sa M^*, e)$.
It follows by saturation and the definition of a $T$-valued Uryshon space that the inclusion $X \hookrightarrow X^*$ extends to an isometric embedding $Y \hookrightarrow X^*$.

\medskip
We now show that $\tu$ admits quantifier elimination.
By model completeness is it enough to show that $(\tu)_\forall$ has the amalgamation property.
Note that $(\tu)_\forall$ is the theory of $T_\forall$-valued metric spaces.
Let $(X, \Sa M, d)$ be a $T_\forall$-valued metric space and let $(X_i, \Sa M_i, d_i)$ be a $T_\forall$-valued metric space extending $(X, \Sa M, d)$ for $i = 1, 2$.
Now, $T_\forall$ has the amalgamation property as $T$ admits quantifier elimination.
Hence there is a model $\Sa N \models T_\forall$ and embeddings $\Sa M_i \hookrightarrow \Sa N$ such that the resulting square commutes.
We may therefore consider each $d_i$ as a function $X^2_i \to N$ and consider the $T_\forall$-valued metric spaces $(X_i, \Sa N, d_i)$.
We may also suppose that $X_1 \cap X_2 = X$.
It is enough to produce a metric $d_\cup$ on $X_\cup = X_1 \cup X_2$ taking values in an elementary extension of $\Sa N$ such that $d_\cup$ agrees with each $d_i$ on $X_i$.
By an obvious compactness argument we may suppose that $X_1, X_2$ are both finite.
We define an $\Sa N$-valued metric $d_\cup$ on $X_\cup$.
Given $a, b \in X_\cup$ we let
\begin{enumerate}[leftmargin=*]
\item $d_\cup(a, b) = d_i(a, b)$ when $a, b \in X_i$, and
\item $d_\cup(a, b) = \min\{ d_1(a, c) + d_2(c, b) : c \in X \}$ when $a \in X_1 \setminus X$, $b \in X_2 \setminus X$.
Note that the minimum exists by finiteness.
\end{enumerate}
Then $(X_\cup, d_\cup)$ is easily seen to be an $\Sa N$-valued metric space~\cite[Def.~5.10]{Conant_distance}.

\medskip
It remains to show that $\tu$ is complete.
By model completeness of $\tu$ it is enough to show that any two models of $(\tu)_\forall
$ jointly embed into a third.
Let $(X_i, \Sa M_i, d_i)$ be a $T_\forall$-valued metric space for $i = 1, 2$.
As $T$ is complete we may suppose that $\Sa M_1, \Sa M_2$ are both substructures of some $\Sa N \models T_\forall$.
As above we consider the $\Sa N$-valued metric spaces $(X_i, \Sa N, d_i)$ for $i = 1, 2$.
After possibly replacing $\Sa N$ with an elementary extension we suppose that there is $\delta \in N$ which is greater than every element of $M_1 \cup M_2$.
We may suppose that $X_1, X_2$ are disjoint.
Let $d$ be the $\Sa N$-valued metric on $X_1 \cup X_2$ such that $d$ agrees with $d_i$ on each $X_i$ and $d(a,b) = \delta$ when $a \in X_1$ and $b \in X_2$.
It is easy to see that $d$ is an $\Sa N$-valued metric.
\end{proof}

A point variable is a variable ranging over the point sort and a distance variable is a variable ranging over the distance sort.
Given a tuple $x = (x_1, \ldots, x_n)$ of point variables we let $\Delta(x)$ be the tuple consisting of all terms $d(x_i, x_j)$.

\begin{corollary}\label{cor:ury}
Let $T$ be as above.
Then $\tu$ is complete and every formula $\phi$ in point variables $x  = (x_1, \ldots, x_n)$ and distance variables $y = (y_1, \ldots, y_m)$ is equivalent to a formula of the form $\vartheta(\Delta(x), y)$ where $\vartheta(z_1, \ldots, z_{n^2 + m})$ is an $L$-formula.
\end{corollary}

Corollary~\ref{cor:ury} follows from Lemma~\ref{lem:ury model comp} by Morleyization.
We now prove Proposition~\ref{prop:ury}, i.e. show that $\tu$ is trace equivalent to $D_2(T)$.

\begin{proof}
Corollary~\ref{cor:ury} shows that $\tu$ is $2$-trace definable in $T$, hence $\tu$ is trace definable in $D_2(T)$.
We show that $\tu$ trace defines $D_2(T)$.
We apply Lemma~\ref{lem;f 0}.
Fix $\Sa O \models T$, a set $P$, and a function $f \colon P^2 \to O$.
Let $\Sa N$ be an elementary extension of $\Sa O$ such that $O$ is not cofinal in $N$.
Fix an element $\delta$ of $N$ which is greater than any element of $O$.
Let $X = P \times \{1, 2\}$ and let $d \colon X^2 \to N$ be given by the following.
\begin{enumerate}[leftmargin=*]
\item $d((p_1, i_1), (p_2, i_2)) = 0$ if $i_1 = i_2$ and $p_1 = p_2$.
\item $d((p_1, i_1), (p_2, i_2)) = \delta$ if $i_1 = i_2$ and $p_1 \ne p_2$.
\item $d((p_1, i_1), (p_2, i_2)) = f(p_1, p_2) + \delta$ if $i_1 \ne i_2$.
\end{enumerate}
Note that $d$ is a $\Sa N$-valued metric on $X$.
Now let $(Y, \Sa N^*, d)$ be a model of $\tu$ extending $(X, \Sa N, d)$.
Finally, apply Lemma~\ref{lem;f 0} with $g \colon Y^2 \to N^*$ given by $g(a, b) = d(a, b) - \delta$ and $a^i_p = (p, i)$ for each $i = 1, 2$ and $p \in P$.
\end{proof}

\section{A non-example}\label{section:non example}
Let $p$ range over primes and fix $k \ge 2$.
An \textbf{extra-special $p$-group} is a group $\Gamma$ such that $a^p = 1$ for all $a \in \Gamma$, the center and commutator of $\Gamma$ agree, and the center of $\Gamma$ is  cyclic  of order $p$.
We let $\esp$ be the theory of infinite extra-special $p$-groups.
An extra-special $2$-group is abelian, so we only consider the case when $p$ is odd.
Felgner showed that $\esp$ is $\aleph_0$-categorical and hence complete~\cite{Felgner}.
This was the first ``algebraic" example of a strictly $2$-$\nip$ structure~\cite{hempel-field}.

\begin{thm}\label{thm:extra special}
Let $p$ be an odd prime.
Then $\esp$ is not locally trace equivalent to $D_k(T)$ for any theory $T$ and $k \ge 2$.
\end{thm}

This is a consequence of a series of results which are interesting on their own.
If $k > 2$ then $D_k(T)$ is $2$-$\ip$ and is hence not locally trace definable in $\esp$.
So it is enough to handle the case $k = 2$.
Recall that $\vvec_p \sqcup \hyp_k$ is the theory of the disjoint union of an infinite $\F_p$-vector space with a model of $\hyp_k$.

\begin{proposition}\label{prop:final final}
Let $p$ be an odd prime.
Then $\esp$ is trace equivalent $\vvec_p \sqcup \hyp_2$.
\end{proposition}

Hence $\esp$ is, modulo trace equivalence, the simplest $\ip$ theory above $\vvec_p$.

\medskip
Let $T = \Th(\F_p)$.
Proposition~\ref{prop:final final} follows by Proposition~\ref{prop:final} below and the well-known fact that $\esp$ is mutually interpretable with $\mathrm{Alt}^*_{T, 2}$.
We describe the proof of the latter.
First, let $V$ be an infinite $\F_p$-vector space and $\spcrngle$ be a non-degenerate $\F_p$-valued alternating bilinear form on $V$, so $(V, \F_p, \spcrngle) \models \mathrm{Alt}^*_{T, 2}$.
Let $\ast$ be the following binary operation on $V \times \F_p$.
\[
(v, \lambda) \ast (v', \lambda') = \left(v + v', \lambda + \lambda' + \frac{1}{2} \langle v, v' \rangle \right) \quad \text{for all $(v, \lambda), (v', \lambda') \in V \times \F_p$}
\]
Then $(V \times \F_p; \ast) \models \esp$.
Conversely, let $\Gamma$ be an infinite extra-special $p$-group, identify the center of $\Gamma$ with $\F_p$, and note that $\Gamma/\F_p$ is an infinite $\F_p$-vector space.
Finally, note that the commutator $\spcbrckt$ on $\Gamma$ induces a non-degenerate alternating form $(\Gamma/\F_p) \times (\Gamma/\F_p) \to \F_p$.

\begin{proposition}\label{prop:final}
Suppose that $\F$ is a finite field of characteristic $p$ and let $T = \Th(\F)$.
Then $\altT$ is trace equivalent to $\vvec_p \sqcup \hyp_k$.
\end{proposition}

It follows that $\altT$ is $k$-trace definable in $\vvec_p$ and in particular $\esp$ is $2$-trace definable in $\vvec_p$.
This disjoint union decomposition is non-trivial: $\vvec_p$ cannot trace define $\hyp_k$ as $\hyp_k$ is unstable and $\hyp_k$ cannot trace define $\vvec_p$ by Theorem~\ref{thm:fh} below.

\medskip
We apply the following easy fact whose verification is left to the reader.

\begin{fact}\label{fact:vect ext}
Suppose that $\E/\F$ is a finite extension of fields.
Then the theory of infinite $\E$-vector spaces is interpretable in the theory of infinite $\F$-vector spaces.
\end{fact}

\begin{proof}[Proof of Proposition~\ref{prop:final}]
We first show that $\altT$ trace defines $\vvec_p \sqcup \hyp_k$.
By Fact~\ref{lem:k-dis} it is enough to show that $\altT$ trace defines both $\vvec_p$ and $\hyp_k$.
First note that $\vvec_p$ is a reduct of $\altT$.
Secondly, the proof of Proposition~\ref{prop;ccs} shows that the formula $\beta(x_1, \ldots, x_k) = 0$ in $\altT$ is $(k - 1)$-$\ip$.
Hence $\altT$ trace defines $\hyp_k$ by Fact~\ref{fact:new hyp rel}(3).

\medskip
We show that $\altT$ is trace definable in $\vvec_p \sqcup \hyp_k$.
Recall that $\rel_k$ is the theory of the generic $k$-ary relation.
Now $\hyp_k$ is trace equivalent to $\rel_k$ by Fact~\ref{fact:new hyp rel}(1), so it is enough to show that $\altT$ is trace definable in $\vvec_p \sqcup \rel_k$.
Fix $(V, \F, \beta) \models \altT$.
So $(V, \F) \models \vvec_T$ for $T = \Th(\F)$.
Let $(M; R)$ be a $|V|^+$-saturated model of $\rel_k$.
Now $(V, \F)$ is mutually interpretable with an $\F$-vector space and by Fact~\ref{fact:vect ext} any $\F$-vector space is interpretable in an $\F_p$-vector space.
Hence it suffices to show that $(V, \F, \beta)$ is trace definable in $(V, \F) \sqcup (M; R)$.

\medskip
Any $k$-ary relation on a set of cardinality $|V|$ embeds into $(M; R)$.
For each $\gamma \in \F$ fix an injection $\uptau_\gamma \colon V \hookrightarrow M$ such that we have
\[
\beta(v_1, \ldots, v_k) = \gamma \quad \Longleftrightarrow \quad R(\uptau_\gamma(v_1), \ldots, \uptau_\gamma(v_k)) \quad \text{for all $v_1, \ldots, v_k \in V$.}
\]
We show that the $\uptau_\gamma$, together with the identities $V \to V$ and $\F \to \F$, witness trace definability of $(V, \F, \beta)$ in $(V, \F) \sqcup (M; R)$.
Let $v = (v_1, \ldots, v_n)$ range over $V^n$ and $c = (c_1, \ldots, c_m)$ range over $\F^m$.
Let $X$ be a $(V, \F, \beta)$-definable subset of $V^n \times \F^m$.
Then there is $d$ and a formula $\vartheta(x_1, \ldots, x_{n + m + d})$ from $(V, \F)$ such that we have $(v, c) \in X$ if and only if
\[
\vartheta(v, c, \beta(v_{i_{1, 1}}, \ldots, v_{i_{1, k}}), \ldots, \beta(v_{i_{d, 1}}, \ldots, v_{i_{d, k}}))
\]
where  the indices $i_{j, l}$ are from $\nset$.
Now let  $b = (b_1, \ldots, b_d)$ range over $\F^d$.
Each $\vartheta(x_1, \ldots, x_n, y_1, \ldots, y_m, b)$ is a formula in $(V, \F)$.
We have $(v, c) \in X$ if and only if
\[
\bigvee_{b} \left( \vartheta(v, c, b)  \land \bigwedge_{j = 1}^{d} \beta(v_{i_{j, 1}}, \ldots, v_{i_{j, k}}) = b_j \right).
\]
Equivalently we have $(v, c) \in X$ if and only if 
\[
\bigvee_{b} \left( \vartheta(v, c, b) \land \bigwedge_{j = 1}^{d} R(\uptau_{bj}(v_{i_{j, 1}}), \ldots, \uptau_{b_j}(v_{i_{j, k}}))  \right).
\]
Finally, note that
\[
\bigvee_{b} \left( \vartheta(x_1, \ldots, x_n , y_1, \ldots, y_m, b) \land \bigwedge_{j = 1}^{d} R(z_{j, 1}, \ldots, z_{j, k})   \right).
\]
is a formula in $(V, \F) \sqcup (M; R)$, here the $x_i$ are variables of sort $V$, the $y_i$ are variables of sort $\F$, and the $z_{j, i}$ are variables of sort $M$.
\end{proof}

We now show that $\esp$ cannot locally trace define $D_2(T)$ for any theory $T$ with infinite models.
It suffices to treat the case $T = \triv$.
Proposition~\ref{prop:dkk triv} shows that $D_2(\triv)$ is trace equivalent to a theory admitting quantifier elimination in a finite relational language.
Hence by Lemma~\ref{lem:ka}(5) a theory locally trace defines $D_2(\triv)$ if and only if it trace defines $D_2(\triv)$.
So by Proposition~\ref{prop:final final} it suffices to show that $D_2(\triv)$ is not trace definable in the theory of a disjoint union of a stable structure with a binary structure.
We apply an ad hoc property which we call partition-wise $\kappa$-stability.
We first define this property.

\medskip
Let $\Sa M$ be a structure.
A {\bf unary expansion} $(\Sa M, \Cal U)$ of $\Sa M$ is an expansion of $\Sa M$ by a collection $\Cal U$ of subsets of $M$.
Given such a collection we let $\Cal U[n]$ be the collection of subsets of $M^n$ of the form $U_1 \times \cdots \times U_n$ for $U_1, \ldots, U_n \in \Cal U$.
Given a unary expansion $(\Sa M, \Cal U)$, $A \subseteq M$, and $n \ge 1$ we let $\nabla^n_A(\Sa M, \Cal U)$ be the Stone space of the boolean algebra of subsets of $M^n$ generated by $A$-definable sets and $\Cal U[n]$.
If $\Sa M$ admits quantifier elimination in a relational language then $\nabla^n_A(\Sa M, \Cal U)$ is the collection of quantifier-free $n$-types in $(\Sa M, \Cal U)$ over $A$.
A {\bf $\Cal U[n]$-type} is an element of the Stone space of the boolean algebra of subsets of $M^n$ generated by $\Cal U[n]$.
Note that any $\Cal U[n]$-type can be uniquely identified with an $n$-tuple of $\Cal U[1]$-types.
Given a $\Cal U[n]$-type $p$ we let $\nabla^n_A(\Sa M, p)$ be the set of $q \in \nabla^n_A(\Sa M, \Cal U)$ extending $p$.
Given an infinite cardinal $\kappa$ we say that a theory $T$ is {\bf partition-wise $\kappa$-stable} if for every $\Sa M \models T$ and $A \subseteq M$ of cardinality at most $\kappa$ there is a collection $\Cal U$ of at most $\kappa$ subsets of $M$ such that $|\nabla^n_A(\Sa M, p)| \le \kappa$ for all $n \ge 1$ and $\Cal U[n]$-types $p$.
We say that such $\Cal U$ witnesses partition-wise $\kappa$-stability for $A$.
Finally, a structure is partition-wise $\kappa$-stable when its theory is.

\medskip
Note that a $\kappa$-stable theory is partition-wise $\kappa$-stable.
We say that a structure or theory is {\bf binary} if it every formula is equivalent to a boolean combination of binary formulas.

\begin{lemma}\label{lem:bin}
Suppose that $T$ is binary.
Then $T$ is partition-wise $\kappa$-stable for any $\kappa \ge 2^{|T|}$.
\end{lemma}

\begin{proof}
After possibly Morleyizing we suppose that $T$ admits quantifier elimination in a binary relational language $L$.
We prove the following claim.
\begin{Claim*}
\label{lem:equal}
Suppose $\Sa M \models T$, $A \subseteq M$, $p_i(x_i)$ is a complete one-type over $A$ for $i = 1, \ldots, k$, and $\varphi(x_1,\ldots,x_k)$ is an $L(A)$-formula.
Then there is an $L$-formula $\varphi^*(x_1,\ldots,x_k)$ such that $$p_1(x_1)\cup\cdots\cup p_k(x_k)\models [\varphi(x_1,\ldots,x_k)\Longleftrightarrow \varphi^*(x_1,\ldots,x_k)].$$
\end{Claim*}

Hence there are at most $2^{|T|}$ types $p(x_1, \ldots, x_k)$ over $A$ extending $p_1(x_1) \cup \cdots \cup p_k(x_k)$.
Therefore the lemma follows by letting $\Cal U$ be the collection of all $A$-definable subsets of $M$.

\medskip
We now prove the claim.
Set $x = (x_1,\ldots,x_k)$ and $p(x)=p_1(x_1) \cup \cdots \cup p_k(x_k)$.
Let $i, j$ range over $\{1, \ldots, k\}$.
By quantifier elimination $\varphi(x)$ is equivalent to a boolean combination of $L(A)$-formulas $\theta(x_i)$  and binary formulas $R(x_i, x_j)$ for $R \in L$.
Now $p(x)$ determines the truth value of each $\theta(x_i)$, so $\varphi(x)$ is equivalent modulo $p(x)$ to a boolean combination of the $R(x_i, x_j)$.
In particular $\varphi(x)$ is equivalent to an $L$-formula.
\end{proof}

\begin{lemma}\label{lem:p preserved}
Let $\kappa$ be an infinite cardinal and $T, T^*$ be theories of cardinality at most $\kappa$.
Suppose that $T$ is partition-wise $\kappa$-stable and $T^*$ is trace definable in $T$.
Then $T^*$ is also partition-wise $\kappa$-stable.
\end{lemma}

\begin{proof}
Fix $\Sa M \models T^*$ and a set $A$ of parameters from $\Sa M$ of cardinality at most $\kappa$.
Fix $\Sa N \models T$ which trace defines $\Sa M$.
We may suppose that $M$ is a subset of $N^m$ and that $\Sa N$ trace defines $\Sa M$ via the inclusion $M \hookrightarrow N^m$.
Let $B \subseteq N$ be such that every $A$-definable subset of every $M^m$ is of the form $Y \cap M^m$ for some $B$-definable $Y \subseteq N^{mn}$.
We may suppose that $|B| \le \kappa$.
Let $\Cal U$ be a collection of subsets of $N$ witnessing partition-wise $\kappa$-stability for $B$.
Let $\Cal V$ be the collection of subsets of $M$ of the form $U \cap M$ for $U \in \Cal U[m]$.
Fix $n \ge 1$ and a $\Cal V[n]$-type $p$.
Note that $p$ determines a $\Cal U[nm]$-type $p^*$.
Suppose $q \in \nabla^n_A(\Sa M, p)$.
We associate to $q$ an incomplete $nm$-type $r$ in $\Sa N$ by letting $r$ be the collection of all $B$-definable subsets $Y$ of $N^{nm}$ such that $Y \cap M^n$ is in $q$.
Now $r \cup p^*$ is consistent and hence extends to an element of $\nabla^{nm}_B(\Sa N, p^*)$.
Therefore we have $|\nabla^n_A(\Sa M, p)| \le |\nabla^{nm}_B(\Sa N, p^*)| \le \kappa$.
\end{proof}

\begin{lemma}\label{lem:p disjoint}
Fix an infinite cardinal $\kappa$.
Then the collection of partition-wise $\kappa$-stable structures is closed under finite disjoint unions.
\end{lemma}

Lemma~\ref{lem:p disjoint} is easy and left to the reader.
(Simply take the disjoint union of the families of sets witnessing partition-wise $\kappa$-stability.)



\medskip
We now consider the theory of a certain ternary finitely homogeneous structure.
Let $L$ be the two-sorted language containing a ternary relation $E$.
Let $T_\mathrm{Feq}$ be the $L$-theory such that an $L$-structure $(M, P; E)$ satisfies $T_\mathrm{Feq}$ if and only if:
\begin{enumerate}
\item $E(c, a, a^*)$ implies $c \in P$ and $a, a^* \in M$.
\item $E(c, x, x^*)$ is an equivalence relation on $M$ for all $c \in P$.
\end{enumerate}
Finite models of $T_\mathrm{Feq}$ form a \Fraisse class~\cite[Lemma~6.3]{CR}.
Let $\tfeq$ be the theory of the limit.

\begin{lemma}\label{lem:neg tfeq}
$\tfeq$ is not partition-wise $\kappa$-stable for any infinite cardinal $\kappa$.
Hence $\tfeq$ is not trace definable in the disjoint union of a stable structure with a binary structure.
\end{lemma}

Fact~\ref{fact:new hyp rel}(2) shows that $\tfeq$ is trace definable in $\hyp_3$.
Now a theory is $2$-$\ip$ if and only if it trace defines $\hyp_3$, so it follows that partition-wise $\kappa$-stability implies $2$-$\nip$.

\begin{proof}
Fix $\kappa \ge \aleph_0$.
Let $E^\kappa = E^\kappa_1$.
So $E^\kappa$ is the model companion of the theory of a set equipped with $\kappa$ equivalence relations.
If $(M, P; E) \models \tfeq$, $(c_i)_{i < \kappa}$ is a sequence of distinct elements of $P$, and $E_i$ is the equivalence relation $E(c_i, x, x^*)$ on $M$ for each $i < \kappa$, then $(M; (E_i)_{i < \kappa}) \models E^\kappa$.
So it suffices to show that $E^\kappa$ is not partition-wise $\kappa$-stable.
Let $\Sa M \models E^\kappa$ and let $\Cal U$ be a collection of subsets of $M$.
For each $I \subseteq \kappa$ let $p_I(x_I, y_I)$ be the $2$-type where $E_i(x_I, y_I)$ is in $p_I$ if and only if $i \in I$.
This gives a collection of $2^\kappa$ distinct $2$-types over the empty set.
Let $r$ be the incomplete $(\Sa M, \Cal U)$-type in the variables $\{x_I, y_I : I \subseteq \kappa \}$ given by
\[
r = \left( \bigcup_{I \subseteq \kappa} p_I \right) \cup \{ x_I \in U \Longleftrightarrow y_I \in U  : I \subseteq \kappa, U \in \Cal U \}.
\]
It suffices to show that $r$ is consistent.
This is a consequence of the following result of Guingona and Parnes~\cite[Example~3.20]{Guingona-Parnes}.
Let $\Sa M^*$ be a reduct of $\Sa M$ to a finite sublanguage and $\Cal U^*$ be a finite subset of $\Cal U$.
They showed that any $\Sa A \in \age(\Sa M^*)$ admits an embedding $\eta$ into $\Sa M^*$ such that the image of $\eta$ is either contained in or disjoint from every $U \in \Cal U^*$.
\end{proof}

\begin{proposition}\label{prop:tfe}
$\tfeq$ is $2$-trace definable in $\triv$.
\end{proposition}

\begin{proof}
Fix $(M, P; E) \models \tfeq$.
Let $N$ be any set of cardinality $|M|$.
For each $p \in P$ let $f_p$ be a function $M \to N$ such that $f_p(a) = f_p(a^*)$ if and only if $E(p, a, a^*)$ for all $a, a^* \in M$.
Let $f \colon P \times M \to N$ be given by $f(p, a) = f_p(a)$ and let $\iota_0, \iota_1$ be arbitrary injections $M \hookrightarrow N$, $P \hookrightarrow N$, respectively.
Finally, note that $f, \iota_0, \iota_1$ witness $2$-trace definability of $(M, P; E)$ in the trivial structure on $N$ by quantifier elimination for $\tfeq$.
\end{proof}

Proposition~\ref{prop:tfe} and Lemmas~\ref{lem:bin}, \ref{lem:p preserved}, \ref{lem:p disjoint}, and \ref{lem:neg tfeq} together show that $D_2(\triv)$ is not locally trace definable in the disjoint union of a stable structure with a binary structure.
In particular $\esp$ does not locally trace define $D_2(\triv)$.
To complete the proof of Theorem~\ref{thm:extra special} we need to show that $\esp$ is not locally trace definable in $D_k(T)$ for $T$ the theory of a finite structure.
We first show that this property can be characterized in several ways.

\medskip
We say that $\Sa M$ is $\infty$-trace definable in $\Sa N$ if $\Sa M$ is $k$-trace definable in $\Sa N$ for some $k \ge 1$.
Note that $\infty$-trace definability is transitive by Proposition~\ref{prop:trace-basic}.

\begin{corollary}\label{cor:hyp rel}
The following are equivalent for any structure $\Sa M$.
\begin{enumerate}[leftmargin=*]
\item $\Sa M$ is $\infty$-trace definable in $\triv$.
\item $\Sa M$ is trace definable in $\hyp_m$ for some $m \ge 2$.
\item $\Sa M$ is $\infty$-trace definable in some theory which admits quantifier elimination in a bounded arity relational language.
\item $\Sa M$ is $\infty$-trace definable in a finite structure.
\end{enumerate}
\end{corollary}

\begin{proof}
Note that (4) implies (1) as any finite structure is interpretable in $\triv$.
Proposition~\ref{prop:dkk triv} shows that $D_k(\triv)$ is trace equivalent to a theory which admits quantifier elimination in a finite $2k$-ary relational language.
By Fact~\ref{fact:new hyp rel}(2) any such theory is trace definable in $\hyp_{2k}$, so (1) implies (2).
It is clear that (2) implies (3).
Proposition~\ref{prop:fhkd} and transitivity of $\infty$-trace definability show that (3) implies (4).
\end{proof}

Theorem~\ref{thm:fh} finishes the proof of Theorem~\ref{thm:extra special}.

\begin{thm}
\label{thm:fh}
Suppose that $\Gamma$ is an infinite group such that $\Gamma$ contains an abelian subgroup of cardinality $\ge n$ for all $n$.
Then $\Gamma$ is not $\infty$-trace definable in $\triv$.
Hence $\triv$ cannot $\infty$-trace define a group of infinite exponent or an infinite locally finite group.
\end{thm}

The last claim of Proposition~\ref{thm:fh} follows from the first by Philip Hall's theorem that an infinite locally finite group has an infinite abelian subgroup~\cite[14.3.7]{Robinson_1996}.
Of course by compactness  Proposition~\ref{thm:fh} is equivalent to the assertion that an infinite group with an infinite abelian subgroup is not $\infty$-trace definable in $\triv$.

\medskip
The existence of infinite  groups with a uniform finite upper bound on the cardinality of an abelian subgroup is essentially equivalent to the failure of the Burnside problem, i.e. existence of infinite finitely generated groups of bounded exponent.
Suppose that $\Gamma$ is an infinite group and every abelian subgroup of $\Gamma$ has $\le n$ elements.
It follows that $\Gamma$ has finite exponent and that there is $m$ such that every finite subgroup of $\Gamma$ has cardinality $\le m$~\cite[pg.~2979]{oger}.
Hence any subgroup of $\Gamma$ generated by $d > m$ distinct elements is infinite, finitely generated, and has finite exponent.
Conversely, let $B(m,n)$ be the free Burnside group of exponent $n$ on $m$ generators, i.e. the quotient of the free group on $m$ generators by the subgroup generated by all $n$th powers. 
If $m\ge 2$ and $n$ is odd and sufficiently large then $B(m,n)$ is infinite and every abelian subgroup is cyclic and so has cardinality $\le m$~\cite{Novikov_1968}.

\medskip
Macpherson showed that  finitely homogeneous structures cannot interpret  infinite groups~\cite{macpherson-interpreting-groups}.
He notes that any group interpretable in a finitely homogeneous structure is $\aleph_0$-categorical, recalls that any infinite $\aleph_0$-categorical group contains an infinite subgroup which is isomorphic to a vector space over a finite field, and then  obtains a contradiction via an application of the affine Ramsey theorem.
This proofs breaks down on the first step for us as $\aleph_0$-categoricity is not preserved under trace definability.
However his proof shows that a finitely homogeneous structure cannot locally trace define an infinite vector space over a finite field.
A theory $T$ is $(m,n)$-homogeneous if whenever $\Sa M\models T$, $\alpha_1,\beta_1,\ldots,\alpha_n,\beta_n\in M$, and we have 
$$ \tp_{\Sa M}(\alpha_{i_1},\ldots,\alpha_{i_m}) =\tp_{\Sa M}(\beta_{i_1},\ldots,\beta_{i_m})\quad\text{for all}\quad 1\le i_1<\ldots<i_m\le n$$
then $\tp_{\Sa M}(\alpha_1,\ldots,\alpha_n)=\tp_{\Sa M}(\beta_1,\ldots,\beta_n)$.
Kikyo~\cite{kikyo} conjectured that the theory of an expansion of an infinite group cannot be $(m,n)$-homogeneous for any $2\le m<n$.
Oger~\cite{oger} used a strong form of Ramsey's theorem due to Leeb-Graham-Rothschild to show that if $\Th(\Gamma)$ is $(m,n)$-homogeneous for some $2\le m<n$ then there is $d\in\N$ such that every abelian subgroup of $\Gamma$ has $\le d$ elements.
Theorem~\ref{thm:fh} follows easily by adapting Oger's proof.

\medskip
We let $\qftp_{\Sa M}(a)$ be the quantifier-free type of a tuple $a$ from $\Sa M$ over the empty set.
Given $b = (b_1, \ldots, b_m)$ and $I = \{i_1, \ldots, i_n\}$ where $1 \le i_1 < \ldots < i_n \le k$ we let $b_I = (b_{i_1}, \ldots, b_{i_n})$

\begin{proof}[Proof of Proposition~\ref{thm:fh}]
The reader will need a copy of \cite{oger}.
Let $(\Gamma;\ast)$ be an infinite group and suppose that $\Gamma$ contains abelian subgroups with cardinality exceeding any given 
integer. 
Suppose towards a contradiction that $(\Gamma;\ast)$ is $\infty$-trace definable in $\triv$.
Hence $(\Gamma; \ast)$ is $k$-trace definable in a finite structure $\Sa M$ for some $k \ge 1$.
We may suppose that $\Sa M$ is the trivial structure on two elements $p, q$ and let $f_1, \ldots, f_n$ be functions $\Gamma^k \to \{p, q\}$ witnessing $k$-trace definability of $(\Gamma ; \ast)$ in $\Sa M$.
We may suppose that the $f_i$ are closed under permutations of variables.
For each $i = 1, \ldots, n$ and $a \in M^d$, $d < n$ let $R_{i, a}$ be the $(k - d)$-ary relation on $\Gamma$ given by declaring $R_i(b_1, \ldots, b_{k - d})$ if and only if $f_i(b_1, \ldots, b_{k - d}, a) = p$.
Let $\Sa G$ be the resulting relational structure on $\Gamma$ and $L$ be its language.
It follows that every $(\Gamma;\ast)$-definable set is quantifier-free definable in $\Sa G$ without parameters.
It follows that if $\alpha,\beta\in \Gamma^n$ then $\qftp_{\Sa G}(\alpha)=\qftp_{\Sa G}(\beta)$ implies $\tp_{(\Gamma;\ast)}(\alpha)=\tp_{(\Gamma;\ast)}(\beta)$.
Let $L^*=L\cup \{\ast\}$ and let $\Sa G^*$ be the natural $L^*$-structure on $\Gamma$.
After possibly passing to an elementary expansion we may suppose that $\Sa G^*$ is $\aleph_1$-saturated.
Note that $\qftp_{\Sa G}(\alpha)$ still determines $\tp_{(\Gamma ; \ast)}(\alpha)$ for any tuple $\alpha$ from $\Gamma$.
By the proof of the main theorem of \cite{oger} there is $b=(b_1,\ldots,b_{k+1})\in \Gamma^{k+1}$ such that
\begin{enumerate}
[leftmargin=*]
\item $b,\ldots,b_{k+1}$ generates an abelian subgroup of $\Gamma$, and
\item $\qftp_{\Sa G^*}(b_1,\ldots,b_k)=\qftp_{\Sa G^*}(b_1,\ldots,b_{i-1},b_i\ast b_{k+1},b_{i+1},\ldots,b_k)$ for all $i\in\{1,\ldots,k\}$
\end{enumerate}
Let $\beta=(b_1,\ldots,b_k,b_1\ast\cdots\ast b_k)$ and $\beta'=(b_1,\ldots,b_k,b_1\ast\cdots\ast b_k\ast b_{k+1})$.
We show that we have $\qftp_{\Sa G^*}(\beta_I) = \qftp_{\Sa G^*}(\beta'_I)$ for all $I\subseteq\{1,\ldots,k+1\}$ with $|I|=k$.
Fix such $I$.
The case $I=\{1,\ldots,k\}$ is trivial so we may suppose that $I=\{1,\ldots,k+1\}\setminus\{i\}$ for some $i\in\{1,\ldots,k\}$.
By (2) $\qftp_{\Sa G^*}(b_1,\ldots,b_k,b_1\ast\cdots\ast b_{k})$ is equal to
\[
\qftp_{\Sa G^*}(b_1,\ldots,b_{i-1},b_i\ast b_{k+1},b_{i+1},\ldots,b_k,b_1\ast\cdots\ast b_{i-1}\ast(b_i\ast b_{k+1})\ast b_{i+1}\ast\cdots\ast b_k).
\]
Hence
\[
\qftp_{\Sa G}(b_1,\ldots,b_{i-1},b_{i+1},\ldots,b_k,b_1\ast\cdots\ast b_k) = \qftp_{\Sa G}(b_1,\ldots,b_{i-1},b_{i+1},\ldots,b_k,b_1\ast\cdots\ast b_k\ast b_{k+1}).
\]
We have shown that $\qftp_{\Sa G}(\beta_I)=\qftp_{\Sa G}(\beta'_I)$ for any $k$-element subset $I$ of $\{1,\ldots,k+1\}$.
As $\Sa G$ is $k$-ary it follows that $\qftp_{\Sa G}(\beta)=\qftp_{\Sa G}(\beta')$, so by assumption $\tp_{(\Gamma;\ast)}(\beta)=\tp_{(\Gamma;\ast)}(\beta')$.
This is a contradiction as $\tp_{(\Gamma;\ast)}(\beta)$ satisfies $x_{k+1}=x_1\ast\cdots\ast x_k$ and $\tp_{(\Gamma;\ast)}(\beta')$ does not.
\end{proof}

\bibliographystyle{abbrv}
\bibliography{NIP}
\end{document}